\documentclass[12pt]{amsart}

\usepackage{amsthm}
\newtheorem{theorem}{Theorem}[section]
\newtheorem{lemma}[theorem]{Lemma}
\newtheorem{proposition}[theorem]{Proposition}
\theoremstyle{definition}
\newtheorem{definition}[theorem]{Definition}

\theoremstyle{remark}
\newtheorem{remark}[theorem]{Remark}
\newtheorem{example}[theorem]{Example}
\newtheorem{corollary}[theorem]{Corollary}

\counterwithin*{section}{part}

\usepackage{amssymb}
\usepackage{empheq}
\usepackage{stmaryrd}
\usepackage{enumerate}
\usepackage[shortlabels]{enumitem}
\usepackage{calc}
\usepackage{url}
\usepackage{comment}

\newcommand{\norm}[1]{\lVert#1\rVert}
\usepackage{mathtools}
\DeclarePairedDelimiter{\ceil}{\lceil}{\rceil}

\DeclareMathOperator{\argmin}{argmin}
\DeclareMathOperator{\argminimax}{argminimax}
\DeclareMathOperator{\minimax}{minimax}
\DeclareMathOperator{\zer}{zer}
\DeclareMathOperator{\dom}{dom}

\newcommand{\CAT}{{\rm{CAT}(0)}}
\newcommand{\ql}[2]{\left\langle\overrightarrow{#1},\overrightarrow{#2}\right\rangle}

\def\EE{\mathbb{E}}
\def\PP{\mathbb{P}}
                
\usepackage[left=2.0cm,%
                right=2.0cm,%
                top=2.5cm,%
                bottom=3.5cm,%
                headheight=12pt,%
                a4paper]{geometry}%

\begin{document}

\title[Convergence of a stochastic PPA in metric spaces]{Mean-square and sublinear convergence of a stochastic proximal point algorithm in metric spaces of nonpositive curvature}

\author[N. Pischke]{Nicholas Pischke}
\date{\today}
\maketitle
\vspace*{-5mm}
\begin{center}
{\scriptsize 
Department of Computer Science, University of Bath,\\
Claverton Down, Bath, BA2 7AY, United Kingdom.\\
E-mail: nnp39@bath.ac.uk}
\end{center}

\maketitle
\begin{abstract}
We define a stochastic variant of the proximal point algorithm in the general setting of nonlinear Hadamard spaces for approximating zeros of the mean of a stochastically perturbed monotone vector field. Generalizing previous work by P.\ Bianchi, we prove the convergence of this method under a suitable strong monotonicity assumption in (separable) Hilbert-Hadamard spaces, that is assuming that all tangent spaces isometrically embed into Hilbert spaces (covering, but not being limited to, the setting of Hadamard manifolds). Moreover, our convergence proof is fully effective and allows for the construction of explicit rates of convergence for the iteration towards the (unique) solution both in mean and almost surely. These rates are moreover highly uniform, being independent of most data surrounding the iteration, space or distribution. In that generality, these rates are novel already in the context of Hilbert spaces. Sublinear nonasymptotic guarantees under additional second-moment conditions on the Yosida approximates and special cases of stochastic convex minimization are discussed.
\end{abstract}


\section{Introduction}

\subsection{Background and motivation}

One of the fundamental problems in stochastic approximation is solving
\[
\min_{x\in X}\int f(s,x)\,d\mu(s),
\]
for a function $f:E\times X\to (-\infty,+\infty]$ on a probability space $(E,\mathcal{E},\mu)$ and some other target space $X$. Indeed, this problem is widely studied for various classes of spaces $X$ and functions $f$, with particular focus being placed on approximation methods and their complexity, and we refer to \cite{Bertsekas2011,Bertsekas2012,NemirovskiJuditskyLanShapiro2009} and the references therein for various discussions along those lines. If $X$ is a Hilbert space and $f$ is a normal convex integrand (see \cite{RockafellarWets1998}), some of the most prominent methods employed in that context are variants of the well-known stochastic proximal point method, that is the iteration
\begin{equation}
x_{n+1}:=\mathrm{prox}_{\lambda_n}^f(\xi_{n+1},x_n)\tag{$+$}\label{convPPA}
\end{equation}
of random variables over an auxiliary probability space $(\Omega,\mathcal{F},\PP)$, assuming a starting point $x_0\in X$, a sequence of parameters $(\lambda_n)\subseteq (0,\infty)$ with certain growth conditions, a sequence $(\xi_{n+1})$ of random variables $\xi_{n+1}:\Omega\to E$ which are independent and identically distributed (i.i.d.)\ with (common) distribution $\mu$, and writing $\mathrm{prox}_{\lambda}^f(s,x):=\argmin_{y\in X}\left\{ f(s,y)+\frac{1}{2\lambda}\norm{x-y}^2\right\}$ for the proximal map of $f$. This iteration and variants thereof, and in particular their complexities, are widely studied under various assumptions on $f$ (we refer to \cite{Bacak2018,Bertsekas2011,Bertsekas2012,NemirovskiJuditskyLanShapiro2009,RyuBoyd}, among many others), prominently when $f$ is strongly convex where then also fast rates of convergence can be achieved under additional moment conditions.

In a non-probabilistic setting, as is well-known (see e.g.\ \cite{BauschkeCombettes2017}), proximal maps of convex functions are just special instantiations of the general notion of a resolvent of a monotone operator. On this general level of monotone operators, Bianchi \cite{Bianchi2016} (see also \cite{Bianchi2015}) studied a corresponding variant of the proximal point algorithm phrased using resolvents for general monotone operators which are now stochastically perturbed, similar as to $f$ above.

Concretely, let $(E,\mathcal{E},\mu)$ and $(\Omega,\mathcal{F},\PP)$ be probability spaces as before, let $X$ be a separable Hilbert space and $A:E\times X\to 2^X$ be a set-valued map. Under a suitable measurability assumption on $A$ and assuming the maximal monotonicity of $A(s,\cdot)$ (the precise assumptions will be discussed later), Bianchi \cite{Bianchi2016} studied the iteration
\begin{equation}
x_{n+1}:=J_{\lambda_n}(\xi_{n+1},x_n)\tag{$*$}\label{prePPA}
\end{equation}
for an i.i.d.\ sequence $(\xi_{n+1})$ of random variables $\Omega\to E$ with distribution $\mu$ and a suitable sequence of parameters $(\lambda_n)$ as before, where now $J_\lambda(s,x):=(\mathrm{Id}+\lambda A(s,\cdot))^{-1}(x)$.

This process indeed generalizes the method \eqref{convPPA} discussed above by setting $A=\partial f$, with $\partial f$ being the (stochastic) subdifferential of $f$. Moreover, as highlighted in \cite{Bianchi2016}, this method bears resemblance in form to the seminal Robbins-Monro method \cite{RobbinsMonro1951} for finding roots of integral functions $\int V(s,x)\,d\mu(s)$, that is $x_{n+1}:=x_n-\lambda_n V(\xi_{n+1},x_n)$, with similar constants as above, as \eqref{prePPA} can also be equivalently written as $x_{n+1}:=x_n-\lambda_n A_{\lambda_n}(\xi_{n+1},x_n)$ where $A_\lambda(s,x):= (x-J_\lambda(s,x))/\lambda$ is the so-called Yosida approximation of the operator $A$.

While the iteration \eqref{convPPA} approximates a minimizer of the mean of the function in question, the iteration generated by \eqref{prePPA} approximates a zero of the mean operator 
\[
\underline{A}(x):=\int A(s,x)\,d\mu(s),
\]
where the integral refers to the Aumann integral \cite{Aumann1965}. Indeed, as shown in \cite{Bianchi2016}, the weighted averages of the sequence $(x_n)$, that is $\overline{x}_n:=\sum_{k=0}^n\lambda_kx_k/\sum_{k=0}^n\lambda_k$, converge weakly to a zero of $\underline{A}$. Approaching this convergence result faces considerable difficulties, resulting among others in an additional uniform integrability assumption, and the proof given by Bianchi in \cite{Bianchi2016} is very sophisticated.

As illustrated in \cite{Bianchi2016}, these additional uniform integrability assumptions and some of the difficulties of the proof can be circumvented under the assumption of strong monotonicity of $A$, and the convergence can then be improved to the strong convergence of $(x_n)$ towards the (in that context unique) zero of $\underline{A}$. However, even in that context the proof given in \cite{Bianchi2016} is nontrivial, relying on various well-known results from stochastic approximation, like the Robbins-Siegmund theorem on supermartingale convergence, that are not immediately recognized to be effective. As put forward by Bianchi in \cite{Bianchi2016}, while the work \cite{Bianchi2016} is set in a highly general context, ``the price to pay with our approach is the absence of convergence rate certificates''. Indeed, while rates have been given for various special cases, they have (essentially) always focused on the case $A=\partial f$ for some suitable convex function $f$ and the general case remains, to our knowledge, quantitatively untreated already for strongly monotone operators $A$ over Hilbert spaces.

\subsection{The contributions of the present paper and related work}

In the present paper, we augment the strong convergence result given in \cite{Bianchi2016} under a strong monotonicity assumption with explicit rates of convergence both in expectation and almost surely. However, we move considerably beyond simply quantitatively outfitting the results of Bianchi by lifting the algorithm to the general nonlinear setting of Hadamard spaces, that is complete geodesic metric spaces of nonpositive curvature. 

Metric spaces of nonpositive curvature were originally introduced by Alexandrov and are commonly called $\CAT$ spaces, after the work of Gromov. Examples range from Hilbert spaces, $\mathbb{R}$-trees and Hadamard manifolds (i.e.\ complete simply connected Riemannian manifolds of nonpositive sectional curvature), to intricate examples like the Billera-Holmes-Vogtmann tree space prominently used in phylogenetics \cite{BilleraHolmesVogtmann2001}. As illustrated by this plethora of spaces, extending tools and results from convex analysis to such metric contexts is particularly well motivated through applied considerations, not the least of which being the extensive developments of machine learning (where optimization over manifolds and other nonlinear spaces plays a key role, as discussed e.g.\ in \cite{ZhangSra2016}). We refer to the seminal monograph \cite{BridsonHaefliger1999} for a comprehensive overview of $\CAT$ and Hadamard spaces and further refer to \cite{Bacak2014a} for a shorter treatment focused on aspects of convex analysis and optimization and to \cite{AlexanderKapovitchPetrunin2023} for a recent treatment of geodesic metric spaces.

Even defining the results of Bianchi in this general context is rather subtle, requiring a synthesis of a range of different notions and results. Concretely, at first we rely on the theory of monotone vector fields in these general geodesic contexts, as introduced by Chaipunya, Kohsaka and Kumam \cite{ChaipunyaKohsakaKumam2021},\footnote{As mentioned in \cite{ChaipunyaKohsakaKumam2021}, this notion seems to be distinct from the notion of a monotone operator on a $\CAT$ space as introduced by Khatibzadeh and Ranjbar \cite{KhatibzadehRanjbar2017}, relying on a previous notion of dual space for a $\CAT$ space by Kakavandi and Amini \cite{KakavandiAmini2010}.} simultaneously generalizing monotone operators on Hilbert spaces and monotone vector fields on Hadamard manifolds (see Section \ref{sec:opCAT}). These in turn further require various considerations on the geometry of geodesic metric spaces, including in particular the notion of tangent spaces introduced in this general context by Nikolaev \cite{Nikolaev1995}, generalizing the respective central notion from Riemannian manifolds (see Section \ref{sec:tangCAT}). Beyond that, we naturally require a theory of integration in the context of Hadamard spaces which was largely developed in the seminal work of Sturm \cite{Sturm2002,Sturm2003} (see Section \ref{sec:intCAT}). In particular, we require an extension of Sturm's integral to set-valued mappings in this metric context, that is an Aumann-Sturm type integral.

All of these considerations come together to define stochastically perturbed monotone vector fields on (separable) Hadamard spaces (see Section \ref{sec:randOp}) and with that a metric analogue of the stochastic proximal point algorithm of Bianchi (see Section \ref{sec:randPPA}). For this method, we prove a strong convergence result (see Theorem \ref{thm:main}) under a strong monotonicity assumption over (separable) Hilbert-Hadamard spaces, that is assuming that all tangent spaces have flat curvature (or, equivalently, that they isometrically embed into Hilbert spaces), a subclass of Hadamard spaces introduced by Gong, Wu and Yu \cite{GongWuYu2021} in the context of work on the Novikov conjecture and discussed in further detail later on. This flatness of the tangent spaces is in particular used to prove that the integral commutes with the pseudo-Riemannian metric for independent random variables, a property on which the proof crucially relies (see Lemma \ref{indepMetric}).

Our result hence immediately covers both Bianchi's original setting of (separable) Hilbert spaces as well as Hadamard manifolds, in which case our results already seem to be in particular qualitatively novel as, to our knowledge, such a stochastic variant of the general proximal point algorithm in the style of Bianchi was not considered in any kind of nonlinear context before. Based on the generality of Hilbert-Hadamard spaces, the result however also reaches beyond the manifold setting (see Example \ref{exHilbertHadamard} later on).

In that way, our results also extend the previous seminal work of Li, L\'opez and Mart\'in-M\'arquez \cite{LiLopezMartinMarquez2009} on the proximal point method for monotone vector fields in Hadamard manifolds for the first time to the stochastic context, at least in the special case of strong monotonicity. In the special case of the subdifferential of a strongly convex function, which will be discussed throughout, our method in particular reduces to a stochastic proximal point algorithm previously studied by Ba\v{c}\'ak \cite{Bacak2018} (extending \cite{Bacak2014b} as well as \cite{Bacak2013}) as well as Ohta and P\'alfia \cite{OhtaPalfia2015}.\footnote{Indeed, our approach is quite different to that of \cite{Bacak2018,OhtaPalfia2015}. In the context of a strongly convex function, our results dispense of the Lipschitz-like conditions from \cite{Bacak2018,OhtaPalfia2015} at the expense of working over Hilbert-Hadamard spaces. Comparing, and perhaps unifying, these works would prove for interesting future work.}

Also in this generalized setting, our convergence proof is in fact fully effective and allows for the construction of explicit rates of convergence for the iteration towards the (unique) solution both in expectation and almost surely, which are highly uniform, being independent of most data surrounding the iteration, space or distribution. Even in the context of strong monotonicity assumptions, rates of convergence for the stochastic proximal point method are largely restricted to the setting of a convex function (where the assumption translates to a strong convexity assumption), such as in the well-known works \cite{AsiDuchi2019,PatrascuNecoara2018} as well as \cite{EisenmannStillfjordWilliamson2022}. The only work known to us that treats general random monotone operators quantitatively is the recent preprint \cite{SadievCondatRichtarik2024} where fast rates of convergence are derived under a strong monotonicity assumption in Hilbert spaces. This work however focuses on the method $x_{n+1}:=J_{\lambda}(\xi_{n+1},x_n)$, i.e.\ where the $\lambda_n$ are kept constant, which in particular results in the method only converging to a neighborhood of the solution. This is hence distinct from the method studied by Bianchi in \cite{Bianchi2016}, which in particular guarantees convergence to the solution, making crucial use of the fact that the $\lambda_n$ vanish (by virtue of assuming $\sum_{n\in\mathbb{N}}\lambda_n^2<\infty$). In that way, in the context of Bianchi's method from \cite{Bianchi2016}, the rates presented in this paper seem to us to be novel already in the context of Hilbert spaces in their generality.

In the end, we briefly discuss applications to the case of the minimization of the expectation of a strongly convex function (see Corollary \ref{cor:convexFct}). Also, we discuss additional uniform boundedness conditions on the second moments of the Yosida approximates, akin to some of the assumptions considered in \cite{AsiDuchi2019,EisenmannStillfjordWilliamson2022} in the special case of a convex function and restricted to a linear setting, which allow us to derive fast nonasymptotic guarantees (see Theorem \ref{thm:mainFast}).\footnote{Over Hilbert-Hadamard spaces, these quantitative results in particular apply to the setting of \cite{Bacak2018,OhtaPalfia2015}, both of which do not provide explicit rates, neither in mean nor almost surely. Instead, the latter only briefly discusses the possibility of constructing fast estimates for the special cases of minimizing a finite sum of strongly convex functions (see \cite[Proposition 5.7]{OhtaPalfia2015} and the discussion thereafter) or of Sturm's strong law of large numbers (see \cite[Remark 6.8]{OhtaPalfia2015}). It should however be noted that \cite{OhtaPalfia2015} is set in the more general setting of spaces with curvature bounded above.}

The methods we employ here to derive the rates of convergence follow a general approach introduced recently by Neri, Powell and the author \cite{NeriPischkePowell2025} (see also \cite{PischkePowell2026}) towards constructing rates of convergence for very general classes of stochastic approximation methods,\footnote{The results from \cite{NeriPischkePowell2025}, and likewise the present results, have been obtained using the logic-based methodology of \emph{proof mining} \cite{Kohlenbach2008,Kohlenbach2019}. More precisely, they are part of a recent advance to apply these logical methods in probability theory and stochastic optimization for the first time \cite{NeriPischke2024,NeriPischkePowell2025,NeriPowell2025,NeriPowell2024,PischkePowell2025,PischkePowell2026}. As common in proof mining however, this paper avoids any reference to mathematical logic.} including ones that pertain to metric generalizations of stochastic quasi-Fej\'er monotonicity (as studied in Hilbert spaces in the seminal works of Combettes and Pesquet \cite{CombettesPesquet2015,CombettesPesquet2019}) of which the present algorithm is, crucially, a particular instance. Indeed, our paper in that way also serves as a case study to illustrate how the abstract approach from \cite{NeriPischkePowell2025} can be used in a very concrete situation to give a perspicuous quantitative analysis of a rather involved algorithm and also how the general metric setting of \cite{NeriPischkePowell2025} can be practically of use. However, the presentation of the paper is self-contained and does not require familiarity with \cite{NeriPischkePowell2025}. We hope that the concrete applications presented in the present paper help to develop future applications of \cite{NeriPischkePowell2025}, such as potentially to the recent works of Karimi, Hsieh, Mertikopoulos and Krause \cite{HsiehKarimiKrauseMertikopoulos2023,KarimiHsiehMertikopoulosKrause2022}, where variants of the Robbins-Monro method over Riemannian manifolds were studied. In particular, the stochastic considerations on metric spaces of the present paper might be of help in lifting these results to a broader metric context.

Also for the stochastic proximal point method studied here, various questions remain which we hope can be answered by future research instigated by the present work, such as whether the weak convergence results from \cite{Bianchi2016} or the convergence results for the distinct method $x_{n+1}:=J_{\lambda}(\xi_{n+1},x_n)$ from the previously mentioned preprint \cite{SadievCondatRichtarik2024} lift to the metric setting.

\section{Preliminaries}\label{sec:prelim}

In this section, we give the necessary preliminaries for the various objects involved in the present paper, such as geodesic metric spaces and their tangent bundles, monotone vector fields on these spaces, and in particular the theory of integration on geodesic spaces of nonpositive curvature. The range of different notions involved, the care needed to bring them together later, and the fact that such a combination has not been considered before, results in rather complex preliminaries which we however have tried to keep as minimal as possible.

\subsection{Geodesics, $\CAT$ spaces and tangent spaces}\label{sec:tangCAT}

We begin with the background on geodesic metric spaces.
Let $(X,d)$ be a metric space. A geodesic is an isometry $\gamma:[0,l]\to X$. We call the image $\gamma([0,l])$ a geodesic segment and say that it joins $x=\gamma(0)$ and $y=\gamma(l)$, as well as that $\gamma$ issues from $x$. Note that necessarily $l=d(x,y)$. $X$ is called (uniquely) geodesic if every two points are joined by a (unique) geodesic. If such geodesics are unique, we denote the (unique) geodesic connecting two points $x,y\in X$ by $\gamma_{x,y}$.

For the purpose of this paper, a geodesic metric space $(X,d)$ is now called a $\CAT$ space (also called a space of nonpositive curvature in the sense of Alexandrov) if it satisfies 
\[
d^2(\gamma(tl),x)\leq (1-t)d^2(\gamma(0),x)+td^2(\gamma(l),x)-t(1-t)d^2(\gamma(0),\gamma(l))
\]
for all $x\in X$ and all geodesics $\gamma:[0,l]\to X$ (that is, an extension of the so-called Bruhat-Tits $\mathrm{CN}$-inequality \cite{BruhatTits1972} to geodesics, see e.g.\ Proposition 2.3 in \cite{Sturm2003}). Any $\CAT$ space is uniquely geodesic. A complete $\CAT$ space is called a Hadamard space. As mentioned in the introduction, we refer to \cite{AlexanderKapovitchPetrunin2023,Bacak2014a,BridsonHaefliger1999} for comprehensive overviews of geodesic metric spaces, $\CAT$ spaces and Hadamard spaces, including alternative definitions.

Another characterization of $\CAT$ spaces that will be useful in this paper was given by Berg and Nikolaev \cite{BergNikolaev2008} using their so-called quasi-inner product (also called quasi-linearization function), that is the map defined by
\[
\ql{xy}{uv}:=\frac{1}{2}\left( d^2(x,v) + d^2(y,u) - d^2(x,u)-d^2(y,v) \right)
\]
for all $x,y,u,v\in X$, where we wrote $\overrightarrow{xy},\overrightarrow{uv}$ as a shorthand for pairs $(x,y),(u,v)\in X^2$. As shown in \cite{BergNikolaev2008}, in any metric space $(X,d)$ this function is the unique function $X^2\times X^2\to\mathbb{R}$ such that for all $x,y,u,v,w\in X$: (1) $\ql{xy}{xy} = d^2(x, y)$; (2) $\ql{xy}{uv}=\ql{uv}{xy}$; (3) $\ql{xy}{uv}=-\ql{yx}{uv}$; (4) $\ql{xy}{uv}+\ql{xy}{vw}=\ql{xy}{uw}$.
It then follows from the results in \cite{BergNikolaev2008} that a geodesic metric space $(X,d)$ is a $\CAT$ space if, and only if,
\begin{equation}
\ql{xy}{uv}\leq d(x,y)d(u,v)\tag{CS}\label{CS}
\end{equation}
for all $x,y,u,v\in X$, i.e.\ where a metric version of the Cauchy-Schwarz inequality holds.

The most important notion for our paper regarding $\CAT$ spaces is that of their tangent spaces as developed in the work of Nikolaev \cite{Nikolaev1995}, as they can be used to provide analogs of fundamental notions from duality theory of linear spaces and manifolds. We essentially follow the exposition and (mostly) the notation of \cite{LewisLopezAcedoNicolae2024} and generally refer to \cite{AlexanderKapovitchPetrunin2023,BridsonHaefliger1999} for further exposition and proofs. Throughout, let $(X,d)$ be a $\CAT$ space. For nonconstant geodesics $\gamma$ and $\eta$ issuing from a point $x\in X$, their Alexandrov angle $\angle_x(\gamma,\eta)$ is defined by
\[
\angle_x(\gamma,\eta):=\lim_{s,t\to 0^+}\bar{\angle}_x(\gamma(s),\eta(t)),
\]
with $\bar{\angle}_x(y,z)$, generally, referring to the comparison angle, defined via the comparison triangle $\bar{\Delta}(\bar{x},\bar{y},\bar{z})$ of the geodesic triangle $\Delta(x,y,z)\subseteq X$, as usual. The Alexandrov angle $\angle_x$ now defines a pseudometric on the set of all nonconstant geodesics issuing from $x$. We write $\Sigma'_xX$ for the set of all equivalence classes of such geodesics under the equivalence relation defined by $\angle_x(\gamma,\eta)= 0$ and we still write $\angle_x$ for the Alexandrov angle extended to these equivalence classes in the obvious way. The completion $(\Sigma_xX,\angle_x)$ of the space $(\Sigma'_xX,\angle_x)$ is called the metric space of directions from $x$ and we denote the elements of it still by letters also used for geodesics, that is $\gamma$, $\eta$, etc. The tangent space $T_xX$ of $X$ at $x$ is then the Euclidean cone over $\Sigma_xX$, that is $T_xX:=(\Sigma_xX\times [0,\infty) )/\sim$ where $(\gamma,t)\sim(\eta,s)$ if, and only if, $t=s=0$ or $t=s>0$ and $\gamma=\eta$. For brevity, we write $t\gamma$ for the equivalence class $[(\gamma,t)]_\sim$ and given $u=t\gamma$ and $\lambda\geq 0$, we write $\lambda u:=(\lambda t)\gamma$. On $T_xX$, we define a metric
\[
d_x(t\gamma,s\eta):=\sqrt{t^2+s^2-2ts\cos\angle_x(\gamma,\eta)}.
\]
The space $TX=\bigcup_{x\in X}T_xX$ is called the tangent bundle of $X$. It in particular follows from the results of Nikolaev \cite{Nikolaev1995} that $T_xX$ is a complete $\CAT$ space, that is a Hadamard space. If $X$ is a Hilbert space, then $T_xX$ reduces to $X$ and if $X$ is a Hadamard manifold, then $T_xX$ reduces to the usual Riemannian tangent space of $X$ at $x$.

We now require some further notation and structure on $T_xX$. We write $0_x:=0\gamma$ and we introduce the notations $\norm{t\gamma}_x:=d_x(0_x,t\gamma)=t$ and
\[
g_x(t\gamma,s\eta):=\frac{1}{2}\left(\norm{t\gamma}_x^2+\norm{s\eta}_x^2-d_x^2(t\gamma,s\eta)\right)=ts\cos\angle_x(\gamma,\eta).
\]
Note that $g_x(t\gamma,s\eta)=\ql{0_xt\gamma}{0_xs\eta}_x$, where we wrote $\ql{\cdot}{\cdot}_x$ for the quasi-inner product on $T_xX$, and so $g_x(t\gamma,s\eta)\leq \norm{t\gamma}_x\norm{s\eta}_x$ by \eqref{CS}, as well as $g_x(t\gamma,t\gamma)=\norm{t\gamma}_x^2$, $g_x(t\gamma,s\eta)=g_x(s\eta,t\gamma)$ and $g_x(t\gamma,s\eta)=tg_x(\gamma,s\eta)$.

Following e.g.\ \cite{Ohta2012} (see also \cite{GongWuYu2021,LeGouicParisRigolletStromme2022,Lammers2024,LytchakNagano2019}), we define the function $\log_x:X\to T_xX$ by $\log_x a:=d(x,a)\gamma_{x,a}$ for $a\neq x$ as well as $\log_x x:=0_x$, which provides an extension of the well-known inverse exponential map, crucial to the study of Riemannian manifolds and their curvature, to this metric setting.\footnote{Indeed, under suitable assumptions on the extendibility of geodesics, one can consider the function $\exp_x t\gamma:=\gamma(t)$ for $\gamma\in \Sigma_xX$, of which $\log_x$ is an inverse of and which provides a metric analog of the exponential map of Riemannian manifolds. We will however not rely on this map in the rest of this paper.} Crucially, note that $\log_x$ is nonexpansive:
\begin{lemma}[{see e.g.\ \cite[eq.\ (2.4)]{LeGouicParisRigolletStromme2022} or \cite[p.\ 224]{GongWuYu2021}}]\label{lognonexp}
For any $x,a,b\in X$:
\[
d_x(\log_xa,\log_xb)\leq d(a,b).
\]
\end{lemma}

The most important property of the pseudo-Riemannian metric $g_x$ on $T_xX$ is the following:

\begin{lemma}[{essentially \cite[Proposition 2.16]{ChaipunyaKohsakaKumam2021}}]\label{tangentCat0}
For any $x,a,b\in X$:
\[
g_x(t\log_x a,s\log_x b)\geq \frac{ts}{2}(d^2(x,a)+d^2(x,b)-d^2(a,b)).
\]
\end{lemma}

Also, we will in the following use that $g_x$ is Lipschitz continuous in both arguments, which we show in the following lemma:

\begin{lemma}\label{nonexpCat0}
For $x\in X$ and $u,v,w\in T_xX$: $\vert g_x(u,v)-g_x(u,w)\vert\leq \norm{u}_xd_x(v,w)$.
\end{lemma}
\begin{proof}
Note that
\begin{align*}
g_x(u,v)-g_x(u,w)&=\frac{1}{2}\norm{v}_x^2-\frac{1}{2}\norm{w}_x^2-\frac{1}{2}d_x^2(u,v)+\frac{1}{2}d_x^2(u,w)\\
&=\frac{1}{2}\left(d^2_x(0_x,v)+d_x^2(u,w)-d_x^2(0_x,w)-d_x^2(u,v)\right)=\ql{0_xu}{wv}_x.
\end{align*}
Using that $T_xX$ is a $\CAT$ space, \eqref{CS} applied to $\ql{0_xu}{wv}_x$ yields
\[
g_x(u,v)-g_x(u,w)=\ql{0_xu}{wv}_x\leq d_x(0_x,u)d_x(v,w)=\norm{u}_xd_x(v,w).
\]
Analogously, we obtain $g_x(u,w)-g_x(u,v)\leq \norm{u}_xd_x(v,w)$. This yields the claim.
\end{proof}

Following \cite{GongWuYu2021}, we call a Hadamard space $X$ a Hilbert-Hadamard space if every tangent space $T_xX$ is flat (see e.g.\ \cite{KhatibzadehRanjbar2017}), that is if the $\mathrm{CN}$-inequality on $(T_xX,d_x)$ is actually an equality, i.e.\ if
\[
d_x^2(\gamma(tl),u)= (1-t)d_x^2(\gamma(0),u)+td_x^2(\gamma(l),u)-t(1-t)d_x^2(\gamma(0),\gamma(l))
\]
for all $u\in T_xX$, $t\in [0,1]$ and all geodesics $\gamma:[0,l]\to T_xX$, for every $x\in X$.

\begin{remark}
The above formulation of Hilbert-Hadamard spaces diverges slightly from the definition used in \cite{GongWuYu2021}, where instead of flatness, it is required that $T_xX$ isometrically embeds into a Hilbert space (which in turn is equivalent to being isometric to a closed and convex subset of a Hilbert space). However, this condition is equivalent to flatness, which follows by a result of Berestovskij and Nikolaev \cite[Theorem 10.3]{BerestovskijNikolaev1993}, a slightly different rendition of which can be found in \cite[Theorem 10.10.13]{BuragoBuragoIvanov2001}. A related exercise can be found in \cite[Advanced Exercise 0.0.2]{AlexanderKapovitchPetrunin2019}, and the result also appears in the unpublished and undated lecture notes of Lurie on Hadamard spaces \cite[Theorem 7.2]{Lurie2026}.
\end{remark}

\begin{example}\label{exHilbertHadamard}
As discussed in Example 3.2 of \cite{GongWuYu2021}, every (separable) Hilbert space and Hadamard manifold is a (separable) Hilbert-Hadamard space, as are (separable) complete simply connected Hilbert manifolds of nonpositive sectional curvature, that is (potentially infinite-dimensional) generalizations of Hadamard manifolds modeled on Hilbert spaces. Further, Hilbert-Hadamard spaces are closed under taking closed and convex subsets. Lastly, Hilbert-Hadamard spaces are closed under ``$L^2$-continuum products'': By \cite[Proposition 3.13, (1)]{GongWuYu2021}, if $(T,\mathcal{T},\tau)$ is a finite measure space and $X$ is a Hilbert-Hadamard space, then so is the space of all square-integrable functions $L^2(T,X,\tau)$ (discussed in detail in Section \ref{sec:intCAT} later on). By \cite[Proposition 3.13, (2)]{GongWuYu2021}, if $X$ and $T$ are further separable,\footnote{A finite measure space $(T,\mathcal{T},\tau)$ is called separable if the metric space $(\mathcal{T},d_\tau)$ with $d_\tau(A,B)=\tau(A\triangle B)$ is separable, where $A\triangle B$ is the symmetric difference of $A,B\in\mathcal{T}$. For example, any probability measure on the Borel space of a Polish metric space induces a separable space.} then so is $L^2(T,X,\tau)$. Using this construction, the class of Hilbert-Hadamard spaces in particular reaches beyond the (infinite-dimensional) manifold setting, including spaces such as the space of $L^2$-Riemannian metrics on a finite dimensional closed smooth manifold (see \cite[Definition 4.2]{GongWuYu2021} and \cite[Remark 9.8]{GongWuYu2021}).
\end{example}

The relevance of the flatness of $T_xX$ for the present paper is that in that context, the pseudo-Riemannian metric $g_x(u,v)$ is affine, that is it is both convex and concave, in both arguments. This follows from a similar argument as used in (the proof of) \cite[Theorem 2.4, (4)]{LeGouicParisRigolletStromme2022}.

\begin{lemma}\label{affine}
Let $x\in X$. If $T_xX$ is flat, then $g_x(u,v)$ is affine in both arguments.
\end{lemma}
\begin{proof}
Fix a geodesic $\gamma:[0,l]\to T_xX$ and $u\in T_xX$. As $T_xX$ is flat, we have
\[
d^2_x(\gamma(tl),u)= (1-t)d^2_x(\gamma(0),u)+td^2_x(\gamma(l),u)-t(1-t)d^2_x(\gamma(0),\gamma(l))
\]
for all $t\in [0,1]$. Using the definition of $g_x$, we hence get 
\begin{align*}
&g_x(\gamma(tl),u)\\
&\qquad=\frac{1}{2}\left(\norm{\gamma(tl)}_x^2+\norm{u}_x^2-d_x^2(\gamma(tl),u)\right)\\
&\qquad= \frac{1}{2}\left((1-t)\norm{\gamma(0)}_x^2+t\norm{\gamma(l)}_x^2+\norm{u}_x^2-
(1-t)d^2_x(\gamma(0),u)-td^2_x(\gamma(l),u)\right)\\
&\qquad=\frac{1-t}{2}\left(\norm{\gamma(0)}_x^2+\norm{u}_x^2-d^2_x(\gamma(0),u)\right)+\frac{t}{2}\left(\norm{\gamma(l)}_x^2+\norm{u}_x^2-d^2_x(\gamma(l),u)\right)\\
&\qquad=(1-t)g_x(\gamma(0),u)+tg_x(\gamma(l),u).
\end{align*}
Hence, $g_x$ is affine in its left argument. By symmetry, $g_x$ is also affine in its right argument.
\end{proof}

\subsection{Monotone vector fields, maximality and resolvents}\label{sec:opCAT}

We now discuss monotone vector fields in this metric context as introduced in \cite{ChaipunyaKohsakaKumam2021},\footnote{The work \cite{ChaipunyaKohsakaKumam2021} relies on a slightly different approach towards the tangent spaces of a $\CAT$ space, which however has no impact on the present paper. All results cited from \cite{ChaipunyaKohsakaKumam2021} hold true in our setup as well.} extending monotone operators on Hilbert spaces and monotone vector fields on Hadamard manifolds as introduced in \cite{Nemeth1999,DaCruzNetoFerreiraLucambioPerez2000} (see also \cite{LiLopezMartinMarquez2009,LiLopezMartinMarquezWang2011,WangLopezMartinMarquezLi2010}). Precisely, a monotone vector field on a $\CAT$ space $X$ is a mapping $A:X\to 2^{TX}$ such that $A(x)\subseteq T_xX$ and 
\[
g_x(u,\log_x y)\leq - g_y(v,\log_y x)
\] 
for all $(x,u),(y,v)\in A$. While not introduced in \cite{ChaipunyaKohsakaKumam2021}, in analogy to \cite{LiLopezMartinMarquez2009} we call the mapping strongly monotone (with modulus $\alpha>0$) if
\[
g_x(u,\log_x y)\leq - g_y(v,\log_y x) - \alpha d^2(x,y)
\] 
for all $(x,u),(y,v)\in A$.

We denote the set of zeros of $A$ by $\zer A:=\{x\in X\mid 0_x\in A(x)\}$. If $A$ is strongly monotone, then it immediately follows that its zero is unique if it exists.

There is a plethora of concrete instantiations that allow one to capture a wide variety of optimization problems as the zero problem of an associated monotone vector field. This in particular includes minimization and saddle point problems, variational inequalities, equilibrium problems and fixed point problems, among many others. For various such examples of monotone operators over Hilbert spaces, we refer to e.g.\ \cite{BauschkeCombettes2017}. For similar constructions over Hadamard manifolds, we again refer to e.g.\ \cite{LiLopezMartinMarquez2009,LiLopezMartinMarquezWang2011,Nemeth1999,DaCruzNetoFerreiraLucambioPerez2000,WangLopezMartinMarquezLi2010}, many of which immediately carry over to the Hadamard space setting. Throughout this paper, we focus on minimization and saddle point problems as two running examples to illustrate the various notions.

\begin{example}\label{subDiffEx}
The canonical example of a monotone vector field is the subdifferential 
\[
\partial f(x):=\{u\in T_xX\mid f(y)\geq f(x)+g_x(u,\log_x y)\text{ for all }y\in X\}
\]
of a proper, convex\footnote{A function $f:X\to (-\infty,+\infty]$ is called convex if $f(\gamma(tl))\leq (1-t)f(\gamma(0))+tf(\gamma(l))$ for any geodesic $\gamma:[0,l]\to X$ and any $t\in [0,1]$.} and lower-semicontinuous function $f:X\to (-\infty,+\infty]$, as defined in the general setting of Hadamard spaces in \cite{ChaipunyaKohsakaKumam2021}. The monotonicity of $\partial f$ follows immediately from the definition, see also \cite[Proposition 3.7]{ChaipunyaKohsakaKumam2021}. Further, it follows from (the proof of) \cite[Corollary 4.5]{LewisLopezAcedoNicolae2024} that
\[
\argmin f:=\{x^*\in X\mid f(x^*)\leq f(x)\text{ for all }x\in X\}=\zer \partial f.
\]
Over Hilbert spaces or Hadamard manifolds, this object naturally reduces to the subdifferential studied there. The recent work of Lewis, L\'opez-Acedo and Nicolae \cite{LewisLopezAcedoNicolae2024} provides an alternative characterization of this object, over locally compact spaces with the geodesic extension property, via normal cones, and establishes further substantial structure theory (even on geodesic metric spaces with general upper bounded curvature). Indeed, that the subdifferential studied in \cite{LewisLopezAcedoNicolae2024} coincides with the one studied in \cite{ChaipunyaKohsakaKumam2021} over suitable Hadamard spaces follows from \cite[Proposition 4.4]{LewisLopezAcedoNicolae2024}.
\end{example}

\begin{example}\label{saddleEx}
Let $X,Y$ be two Hadamard spaces and $l:X\times Y\to\mathbb{R}$ be a saddle function (see e.g.\ \cite{Kohsaka2026,LiLopezMartinMarquez2009}), that is $l(x,\cdot)$ is convex and lower-semicontinuous for all $x\in X$ and $l(\cdot,y)$ is concave (i.e.\ $-l(\cdot,y)$ is convex) and upper-semicontinuous for all $y\in Y$. Similar to \cite{LiLopezMartinMarquez2009} (as well as \cite{Rockafellar1970,Rockafellar1976}), the field $A_l$ defined by
\[
A_l(x,y):=\partial [-l(\cdot,y)](x)\times \partial [l(x,\cdot)](y)
\]
is a monotone vector field over $X\times Y$. To see this, first note that $X\times Y$ is a Hadamard space under the $\ell^2$-product metric $d^2((x,y),(x',y'))=d^2(x,x')+d^2(y,y')$ (see \cite[Example 1.15]{BridsonHaefliger1999}), from which it follows that $T_{(x,y)}(X\times Y)=T_xX\times T_yY$ as well as 
\[
g_{(x,y)}((u,v),(u',v'))=g_x(u,u')+g_y(v,v')\text{ and }\log_{(x,y)}(z,z')=(\log_xz,\log_{y}z').
\]
The monotonicity of $A_l$ then rather immediately follows from the definition of the subdifferential. Further, it holds that $\minimax l=\zer A_l$, where
\[
\minimax l:= \{(x^*,y^*)\in X\times Y\mid l(x,y^*)\leq l(x^*,y^*)\leq l(x^*,y)\text{ for all }(x,y)\in X\times Y\},
\]
which immediately follows from the corresponding property $\argmin f=\zer \partial f$ for a convex function $f$ as discussed in Example \ref{subDiffEx}.
\end{example}

The subdifferential also provides a suitable example for a strongly monotone vector field in the case where $f$ is a strongly convex function with modulus $\alpha>0$, i.e.\ where
\[
f(\gamma(tl))\leq (1-t)f(\gamma(0))+tf(\gamma(l)) -t(1-t)\frac{\alpha}{2}d^2(\gamma(0),\gamma(l))
\]
for any geodesic $\gamma:[0,l]\to X$ and any $t\in [0,1]$, as the following Proposition \ref{strConToStrMon} shows. In particular, this result generalizes a similar result given in \cite{WangLopezMartinMarquezLi2010} in the setting of monotone vector fields over Hadamard manifolds (which also requires a rather involved proof compared to the corresponding result in Hilbert spaces, originally due to Rockafellar \cite{Rockafellar1976}, see also \cite[Example 22.4]{BauschkeCombettes2017}).

\begin{proposition}\label{strConToStrMon}
If $f$ is a strongly convex function with modulus $\alpha>0$, then $\partial f$ is strongly monotone with modulus $\alpha$.
\end{proposition}
\begin{proof}
Assuming w.l.o.g.\ that $x\neq y$, and writing $l=d(x,y)$, note that for $(x,u),(y,v)\in\partial f$, we have
\[
(1-t)f(x)+tf(y)-t(1-t)\frac{\alpha}{2}d^2(x,y)\geq f(\gamma_{x,y}(tl))\geq f(x)+g_x(u,\log_x\gamma_{x,y}(tl))
\]
and so
\[
f(y)-f(x)\geq \frac{g_x(u,\log_x\gamma_{x,y}(tl))}{t}+(1-t)\frac{\alpha}{2}d^2(x,y)
\]
for any $t\in (0,1]$. Similarly, we get
\[
f(x)-f(y)\geq \frac{g_y(v,\log_y\gamma_{y,x}(tl))}{t}+(1-t)\frac{\alpha}{2}d^2(x,y)
\]
for all $t\in (0,1]$ so that combined, using $g_x(u,\log_x\gamma_{x,y}(tl))/t=g_x(u,\log_xy)$ (and similarly for $y$ and $v$), we have
\[
-g_y(v,\log_yx)-(1-t)\alpha d^2(x,y)\geq g_x(u,\log_xy).
\]
Sending $t\to 0$ yields $-g_y(v,\log_yx)-\alpha d^2(x,y)\geq g_x(u,\log_xy)$ by which $\partial f$ is strongly monotone with modulus $\alpha>0$. To see that $g_x(u,\log_x\gamma_{x,y}(tl))/t=g_x(u,\log_xy)$ (and similarly for $y$ and $v$), note that we have $\log_x\gamma_{x,y}(tl)=d(x,\gamma_{x,y}(tl))\gamma_{x,\gamma_{x,y}(tl)}=td(x,y)\gamma_{x,\gamma_{x,y}(tl)}$ so that using Lemma \ref{nonexpCat0}, it holds that
\begin{align*}
\vert \tfrac{1}{t}g_x(u,\log_x\gamma_{x,y}(tl))-g_x(u,\log_xy)\vert &=d(x,y)\vert g_x(u,\gamma_{x,\gamma_{x,y}(tl)})-g_x(u,\gamma_{x,y})\vert\\
&\leq d(x,y)\norm{u}_xd_x(\gamma_{x,\gamma_{x,y}(tl)},\gamma_{x,y})\\
&= d(x,y)\norm{u}_x\sqrt{2-2\cos\angle_x(\gamma_{x,\gamma_{x,y}(tl)},\gamma_{x,y})}.
\end{align*}
Now, by definition we have
\begin{align*}
\angle_x(\gamma_{x,\gamma_{x,y}(tl)},\gamma_{x,y})&=\lim_{s,s'\to 0^+}\bar{\angle}_x(\gamma_{x,\gamma_{x,y}(tl)}(s),\gamma_{x,y}(s'))\\
&=\lim_{s,s'\to 0^+}\bar{\angle}_x(\gamma_{x,y}(s),\gamma_{x,y}(s'))\\
&=\angle_x(\gamma_{x,y},\gamma_{x,y})=0,
\end{align*}
using that $\gamma_{x,\gamma_{x,y}(tl)}(s)=\gamma_{x,y}(stl)$ for suitably small $s$. Hence, we have
\begin{align*}
&\vert \tfrac{1}{t}g_x(u,\log_x\gamma_{x,y}(t))-g_x(u,\log_xy)\vert\\
&\qquad\leq d(x,y)\norm{u}_x\sqrt{2-2\cos\angle_x(\gamma_{x,\gamma_{x,y}(tl)},\gamma_{x,y})}=0.
\end{align*}
\end{proof}

\begin{remark}
Similar arguments as used in the proof of Proposition \ref{strConToStrMon} can also be used to show that the subdifferential of a strongly convex function $f$ with modulus $\alpha$ satisfies
\[
\partial f(x):=\left\{u\in T_xX\mid f(y)\geq f(x)+g_x(u,\log_x y)+\frac{\alpha}{2}d^2(x,y)\text{ for all }y\in X\right\}.
\]
This can then in particular be used to show that the saddle field $A_l$ from Example \ref{saddleEx} is strongly monotone with modulus $\min\{\alpha,\beta\}$ whenever $l$ is a saddle function in the sense of Example \ref{saddleEx} which is also strongly convex in its right argument with modulus $\alpha>0$ and strongly concave (i.e.\ $-l$ is strongly convex) in its left argument with modulus $\beta>0$.
\end{remark}

Following \cite{ChaipunyaKohsakaKumam2021} (which in turn generalizes \cite{LiLopezMartinMarquezWang2011}), we define the resolvent $J_\lambda$ via
\[
J_\lambda x:=\{z\in X\mid \tfrac{1}{\lambda}\log_zx\in A(z)\}.
\]
As shown in (the proof of) \cite[Proposition 3.4]{ChaipunyaKohsakaKumam2021}, it follows from Lemma \ref{tangentCat0} that if $A$ is monotone, then for any $x\in X$ and $\lambda>0$, a $z\in X$ such that $\frac{1}{\lambda}\log_zx\in A(z)$ is necessarily unique. In that case, we identify $J_\lambda$ with the corresponding (potentially partial) function from $X$ to $X$.

A monotone vector field $A$ is called maximal if its graph $\mathrm{gra}A:=\{(x,u)\in X\times TX\mid u\in Ax\}$ cannot be extended properly while preserving monotonicity. By the well-known theorem of Minty \cite{Minty1962}, the maximality of a monotone operator over a Hilbert space is equivalent to the totality of the resolvent. This extends to the setting of Hadamard manifolds provided that $\dom A:=\{x\in X\mid A(x)\neq\emptyset\}=X$ (see \cite[Remark 4.4]{LiLopezMartinMarquez2009}).

As mentioned in \cite{ChaipunyaKohsakaKumam2021}, it is unknown whether this equivalence between maximality and totality extends to this general setting of $\CAT$ spaces. One direction however remains valid, as shown in \cite[Proposition 3.5]{ChaipunyaKohsakaKumam2021}: if the resolvents are all total, that is for any $\lambda>0$ and $x\in X$ there exists a $z\in X$ with $\tfrac{1}{\lambda}\log_zx\in A(z)$, then $A$ is maximal. Following \cite{ChaipunyaKohsakaKumam2021}, we say that   $A$ satisfies the surjectivity condition if all resolvents are total.
 
\begin{example}\label{resEx}
As shown in \cite[Proposition 3.8]{ChaipunyaKohsakaKumam2021}, the resolvent of the subdifferential $\partial f$ of a proper, convex and lower-semicontinuous function $f:X\to (-\infty,+\infty]$ is given by
\[
\mathrm{prox}_{\lambda}^fx:=\argmin_{y\in X}\left\{ f(y)+\frac{1}{2\lambda}d^2(x,y)\right\},
\]
that is the proximal map of $f$ (also called the Moreau-Yosida resolvent), which in this context of Hadamard spaces was first defined by Jost \cite{Jost1995}. In particular, as each $\mathrm{prox}_{\lambda}^f$ is total (see \cite[Lemma 2]{Jost1995}), $\partial f$ satisfies the surjectivity condition (and so is also maximal).
\end{example}

\begin{example}\label{resSaddleEx}
In the case of the saddle field $A_l$ for a saddle function $l:X\times Y\to\mathbb{R}$ over two Hadamard spaces $X$ and $Y$, it follows easily from the definition that its resolvent is given by
\[
\mathrm{R}_{\lambda}^l(x,y):=\argminimax_{(z,w)\in X\times Y}\left\{ l(z,w)-\frac{1}{2\lambda}d^2(z,x)+\frac{1}{2\lambda}d^2(w,y)\right\},
\]
that is the resolvent of the saddle function $l$ as defined in Hadamard spaces in the recent work \cite{Kohsaka2026} (generalizing \cite{LiLopezMartinMarquez2009,Rockafellar1970,Rockafellar1976}). In particular, by \cite[Theorem 5.1]{Kohsaka2026}, each $\mathrm{R}_{\lambda}^l$ is total, so that $A_l$ satisfies the surjectivity condition (and so is also maximal, as was obtained in a different way over Hadamard manifolds in \cite[Theorem 5.6]{LiLopezMartinMarquez2009}).
\end{example}

In terms of essential properties of the resolvent required in this paper, we have the following:

\begin{lemma}[{see \cite[Proposition 4.3]{ChaipunyaKohsakaKumam2021}}]
For any $\lambda>0$, $J_\lambda$ is nonexpansive, i.e.\ 
\[
d(J_\lambda x,J_\lambda y)\leq d(x,y)
\]
for any $x,y\in\mathrm{dom}(J_\lambda)$. Further, $\mathrm{Fix}(J_\lambda)=\zer A$.
\end{lemma}

The resolvent is in fact even firmly nonexpansive in the sense of Ariza-Ruiz, Leu\c{s}tean and L\'opez-Acedo \cite{ArizaRuizLeusteanLopezAcedo2014}, but we will not rely on this here.

At last, we will rely on the so-called Yosida approximate of the field $A$. We define this object here via
\[
A_\lambda x:=\tfrac{1}{\lambda}\log_{J_\lambda x} x.
\]
This seems to be distinct from the variant of the Yosida approximate introduced in \cite{ChaipunyaKohsakaKumam2021}, which relies on the so-called negative geodesics constructed using the geodesic extension property. However, the above definition precisely serves our purpose here as we have the following two crucial properties:

\begin{lemma}\label{YosidaProp}
For any $\lambda>0$ and any $x\in\mathrm{dom}(J_\lambda)$:
\begin{enumerate}
\item $A_\lambda x\in A(J_\lambda x)$,
\item $\norm{A_\lambda x}_{J_\lambda x}=\norm{\tfrac{1}{\lambda}\log_{J_\lambda x} x}_{J_\lambda x}=\frac{1}{\lambda}d(x,J_\lambda x)$.
\end{enumerate}
\end{lemma}
\begin{proof}
The first item is immediate by definition of $J_\lambda$ as if $x\in\mathrm{dom}(J_\lambda)$, then $A_\lambda x=\tfrac{1}{\lambda}\log_{J_\lambda x} x\in A(J_\lambda x)$ since $J_\lambda$ is single-valued. The second item is immediate from the definition of $\norm{\cdot}$ and $\log$, by which we have $\norm{\tfrac{1}{\lambda}\log_{J_\lambda x} x}_{J_\lambda x}=\tfrac{1}{\lambda}d(x,J_\lambda x)$.
\end{proof}

On the contrary, the definition given in \cite{ChaipunyaKohsakaKumam2021} does not seem to satisfy the above essential inclusion given in item (1). The mapping will however only serve a technical and in a way auxiliary purpose here and is not our main object of study so that no compatibility issues between the present work and \cite{ChaipunyaKohsakaKumam2021} arise.

\subsection{Measurability and integration on $\CAT$ spaces}\label{sec:intCAT}

We now introduce the necessary background from the theory of random variables with values in Hadamard spaces. This theory goes back to the seminal work of Sturm and covers a range of advanced areas such as martingales and Markov processes. We here only rely on Sturm's relatively early works \cite{Sturm2002,Sturm2003} for the development of the $L^1$- and $L^2$-theory of these random variables. We also refer to the exposition of these matters given by Ba\v{c}\'ak \cite{Bacak2014a}. For that, let $(T,\mathcal{T},\tau)$ be a probability space and $(X,d)$ be a separable Hadamard space. An ($X$-valued) random variable is a map $x:T\to X$ which is $\mathcal{T}$/$\mathcal{B}(X)$-measurable, where $\mathcal{B}(X)$ is the Borel $\sigma$-algebra of $X$. Write $\mathcal{P}(X)$ for the set of all probability measures on $\mathcal{B}(X)$ and for $p\in [1,\infty)$, write $\mathcal{P}^p(X)$ for the set of all measures $P\in\mathcal{P}(X)$ such that $\int d^p(w,z)\,dP(w)<\infty$ for some/any $z\in X$. As usual, the push-forward measure $\tau_{x}$, given by $\tau_{x}(A):= \tau(x^{-1}(A))$ for $A\in \mathcal{B}(X)$, is called the distribution of $x$. Naturally, we have $\tau_x\in\mathcal{P}(X)$.

Fundamentally (compare e.g.\ \cite[Theorem 2.3.1]{Bacak2014a}), we have the following result in Hadamard spaces: For any $P\in\mathcal{P}^1(X)$ and some/any $y\in X$, there is a unique minimizer 
\[
b(P):=\argmin_{z\in X}\int\left(d^2(z,w)-d^2(w,y)\right)\,dP(w),
\]
which is independent of $y$, called the barycenter of $P$. If $P\in\mathcal{P}^2(X)$, we further have
\[
b(P)=\argmin_{z\in X}\int d^2(w,z)\,dP(w).
\]

Given $p\in [1,\infty)$, the space $\mathcal{L}^p(T,X,\tau)$ is the space of all $\mathcal{T}$/$\mathcal{B}(X)$-measurable maps $x:T\to X$ with $d_p(x,z)<\infty$ for some/any $z\in X$, where $d^p_p(x,y):= \int d^p(x,y)\,d\tau$ for $\mathcal{T}$/$\mathcal{B}(X)$-measurable maps $x,y:T\to X$ and where we identify $z\in X$ with a constant map $T\to X$. The space $L^p(T,X,\tau)$ arises from this space by considering equivalence classes under the equivalence relation defined by $d_p(x,y)=0$. Note that $x\in L^p(T,X,\tau)$ if, and only if, $\tau_x\in\mathcal{P}^p(X)$.

The expectation of a random variable $x\in L^1(T,X,\tau)$ is now defined via 
\[
\EE[x]:=\int x\,d\tau:=b(\tau_x)
\]
where $b(\tau_x)$ is the barycenter of $\tau_x$ as before.

We can similarly define the conditional expectation of a random variable $x:T\to X$ relative to a sub-$\sigma$-algebra $\mathcal{T}_0\subseteq\mathcal{T}$. Concretely, as shown in \cite{Sturm2002}, for any $x\in L^2(T,X,\tau)$ there is a unique equivalence class of $\mathcal{T}_0$/$\mathcal{B}(X)$-measurable random variables $z\in L^2(T,X,\tau)$ which minimizes $d_2(z,x)$. We denote this equivalence class by $\EE[x\,\vert\, \mathcal{T}_0]$. Note that for $\mathcal{T}_0=\{\emptyset,T\}$, this notion reduces to the previously defined expectation. This $L^2$-theory of conditional expectations then continuously extends to the $L^1$-case (see Corollary 2.4 in \cite{Sturm2002}).

Indeed, we require only relatively little theory of the above objects beyond their definitions. The first result is the following transformation theorem (which crucially employs separability):
\begin{lemma}[{see \cite[p.\ 371]{Sturm2003}}]\label{transfer}
Let $(T',\mathcal{T}')$ be another measure space. For any $\mathcal{T}$/$\mathcal{T}'$-measurable $\zeta:T\to T'$ and any $\mathcal{T}'$/$\mathcal{B}(X)$-measurable $x:T'\to X$ such that $x\circ \zeta$ and $x$ are integrable, it holds that
\[
\int x\circ\zeta\,d\tau=\int x\;d\tau_\zeta.
\]
\end{lemma}

The second result that we mention is a version of Jensen's inequality, formulated even for conditional expectations:
\begin{lemma}[{see \cite[Proposition 3.4]{Sturm2002}}]\label{nonlinJensen}
Let $\mathcal{T}_0\subseteq\mathcal{T}$ be a sub-$\sigma$-algebra. Further, let $x\in L^1(T,X,\tau)$ and let $\varphi:X\to\mathbb{R}$ be lower-semicontinuous and convex with $(\varphi\circ x)_+\in L^1(T,\tau)$. Then
\[
\EE[\varphi\circ x\mid\mathcal{T}_0]\geq \varphi\left(\EE[x\mid\mathcal{T}_0]\right).
\]
\end{lemma}

The last result is the following independence property for conditional expectations, which lifts a well-known result from real-valued conditional expectations (see e.g.\ \cite[Theorem 8.14]{Klenke2020}) to the case of Hadamard spaces:

\begin{lemma}\label{indepLem}
Let $\mathcal{T}_0\subseteq\mathcal{T}$ be a sub-$\sigma$-algebra and let $x\in L^1(T,X,\tau)$ be independent of $\mathcal{T}_0$, that is its generated $\sigma$-algebra $\sigma(x)$ is independent of $\mathcal{T}_0$ in the usual sense (see e.g.\ \cite{Klenke2020}). Then
\[
\EE[x\mid\mathcal{T}_0]=\int x\,d\tau.
\]
\end{lemma}
\begin{proof}
We first show the claim for $x\in L^2(T,X,\tau)$. Note that $\int x\,d\tau$ as a constant map is clearly in $L^2(T,X,\tau)$ and $\mathcal{T}_0/\mathcal{B}(X)$-measurable. To see $\EE[x\mid\mathcal{T}_0]=\int x\,d\tau$ and as $x\in L^2(T,X,\tau)$, it thus remains to see that $\int x\,d\tau$ minimizes $d_2(z,x)$ amongst all $\mathcal{T}_0/\mathcal{B}(X)$-measurable $z\in L^2(T,X,\tau)$. For that, simply note that
\[
\int d^2(z,x)\,d\tau=\int\int d^2(z(s'),x(s))\,d\tau(s)\,d\tau(s')\geq \int d^2\left(\int x\,d\tau ,x(s)\right)\,d\tau(s),
\]
where the equality follows as $z$ is $\mathcal{T}_0/\mathcal{B}(X)$-measurable and hence independent of $x$, and the inequality follows from the fact that $\int x\,d\tau=\argmin_{w\in X}\int d^2(x,w)\,d\tau$ as $x\in L^2(T,X,\tau)$. Now, given $x\in L^1(T,X,\tau)$, fix a $z\in X$ and, given $n\in\mathbb{N}$, define $x_n$ following \cite{Sturm2002} via $x_n(s):=x(s)$ if $d(x(s),z)\leq n$ and $x_n(s):=z$ otherwise. Then $x_n\in L^2(T,X,\tau)$ and $\int d(x_n,x)\,d\tau\to 0$ as $n\to\infty$, that is in particular $d(\int x_n\,d\tau,\int x\,d\tau)\leq \int d(x_n,x)\,d\tau\to 0$ by \cite[Corollary 2.4]{Sturm2002} and so $\int x_n\,d\tau\to \int x\,d\tau$. Moreover, as $x$ is independent of $\mathcal{T}_0$, so is each $x_n$. Hence $\EE[x_n\mid\mathcal{T}_0]=\int x_n\,d\tau$ for all $n\in\mathbb{N}$. By definition, we have
\[
\EE[x\mid \mathcal{T}_0]=\lim_{n\to\infty}\EE[x_n\mid\mathcal{T}_0]=\lim_{n\to\infty}\int x_n\,d\tau=\int x\,d\tau,
\]
as claimed.
\end{proof}

As we will later be concerned with these probabilistic notions for random variables taking values not only in a Hadamard space $X$ but also in its tangent spaces $T_xX$, we comment on some measurability aspects and some crucial assumptions related to this already in the following remark:

\begin{remark}\label{rem:tangentSepAndMeas}
As seen above, we crucially rely on the separability and completeness of the underlying geodesic space to develop probability theory over it. As a consequence, we will later throughout use that each tangent space $T_xX$ is separable, given an underlying separable Hadamard space $X$. Indeed, this can be shown as follows: First note that $\Sigma'_xX$ is separable. For that, given a countable and dense $Z\subseteq X$, consider the set 
\[
\gamma_{x,Z}:=\{\gamma_{x,z}\mid z\in Z\text{ and }x\neq z\}\subseteq \Sigma'_xX.
\]
Let now $\gamma$ be a nonconstant geodesic issuing from $x$ with endpoint $y$, that is $\gamma=\gamma_{x,y}$, and take $(z_n)\subseteq Z$ such that $z_n\to y$. Then $\angle_x(\gamma_{x,y},\gamma_{x,z_n})\leq\bar{\angle}_x(y,z_n)\to 0$, where the inequality follows from \cite[Proposition II.1.7]{BridsonHaefliger1999} and the limit follows by the law of cosines applied to $\bar{\angle}_x$ (see \cite[p.\ 9]{BridsonHaefliger1999}). As $\Sigma'_xX$ is thereby separable, so is its completion $\Sigma_xX$. The Euclidean cone over $\Sigma_xX$, that is $T_xX$, is then immediately separable as well.

In any case, if $T_xX$ is separable, then $d_x$ is jointly measurable and so $\norm{\cdot}_x$, $g_x$ and $\ql{xy}{uv}_x$ are (jointly) measurable as well. Further, note that $\log_x$ is measurable as it is nonexpansive and hence uniformly continuous.
\end{remark}

As usual, we say that a sequence $(x_n)$ of ($X$-valued) random variables is independent if the generated $\sigma$-algebras $\sigma(x_n)$ are independent and $(x_n)$ is called identically distributed if all distributions $\tau_{x_n}$ coincide. We abbreviate the property that a sequence is independent and identically distributed by i.i.d., also as usual.

In a separable Hilbert space $(X,\langle\cdot,\cdot\rangle)$, it is well-known that independent ($X$-valued) random variables $u,v:T\to X$ satisfy 
\[
\int\langle u,v\rangle\,d\tau=\left\langle \int u\,d\tau,\int v\,d\tau\right\rangle,
\]
where the integral in that case reduces to the Bochner integral on $X$. We will later require this property for random variables taking values in the tangent space $T_xX$ of a Hadamard space $X$, relative to the mapping $g_x$ considered above. While this therefore naturally holds for Hilbert spaces and Hadamard manifolds, it is not clear if this property holds in this full generality. Indeed, the only partial answer we can give here is that the equality always holds when $T_xX$ has flat curvature, i.e.\ in particular when $X$ is a Hilbert-Hadamard space. Concretely, utilizing that $g_x$ is affine in this case (recall Lemma \ref{affine}), we get the following result:

\begin{lemma}\label{indepMetric}
Fix $x\in X$ and suppose $T_xX$ is flat. Let $\mathcal{T}_0\subseteq\mathcal{T}$ be a sub-$\sigma$-algebra. If $u,v\in L^1(T,T_xX,\tau)$ are such that $v$ is $\mathcal{T}_0$/$\mathcal{B}(T_xX)$-measurable and $u$ is independent of $\mathcal{T}_0$, then we have
\[
\EE[g_x(u,v)\mid \mathcal{T}_0]= g_x\left(\EE[u\mid \mathcal{T}_0],\EE[v\mid \mathcal{T}_0]\right).
\]
\end{lemma}
\begin{proof}
Note that as $u,v$ are integrable and independent, also $g_x(u,v)$ is integrable as we have 
\[
\int \vert g_x(u,v)\vert \,d\tau\leq \int \norm{u}_x\norm{v}_x \,d\tau= \left(\int \norm{u}_x \,d\tau\right)\left(\int \norm{v}_x\,d\tau\right)<\infty,
\]
where the equality follows by independence. We now have 
\[
\EE[g_x(u,v)\mid \mathcal{T}_0](s)=\int g_x(u,v(s))\,d\tau=g_x\left(\int u\,d\tau,v(s)\right)
\]
for almost all $s\in T$, where the first equality follows by independence of $u$ from $\mathcal{T}_0$ (and hence from $v$) and the second follows from Jensen's inequality (recall Lemma \ref{nonlinJensen}), using that $g_x$ is affine (recall Lemma \ref{affine}). Using properties of the conditional expectation for real-valued random variables, we hence have
\[
\EE[g_x(u,v)\mid \mathcal{T}_0]=\EE\left[g_x\left(\int u\, d\tau,v\right)\mid\mathcal{T}_0\right]=g_x\left(\int u\, d\tau,\EE[v\mid\mathcal{T}_0]\right),
\]
where the last equality again follows from Jensen's inequality, using that $g_x$ is affine. This implies
\[
\EE[g_x(u,v)\mid \mathcal{T}_0]=g_x\left(\int u\, d\tau,\EE[v\mid\mathcal{T}_0]\right)=g_x(\EE[u\mid\mathcal{T}_0],\EE[v\mid\mathcal{T}_0])
\]
via Lemma \ref{indepLem}, using that $u$ is independent of $\mathcal{T}_0$.
\end{proof}

As this issue is thereby rather delicate in this metric context, we will always explicitly highlight whenever we assume flatness of the tangent spaces.

We now transfer the above notion of an integral to set-valued maps in an analogous way as the seminal work of Aumann \cite{Aumann1965} did for random variables taking values in (separable) Hilbert spaces, replacing the use of the Bochner integral therein with the integral of Sturm.

Concretely, given a set-valued map $F:T\to 2^X$, a function $\phi:T\to X$ is a measurable selection of $F$ if it is $\mathcal{T}$/$\mathcal{B}(X)$-measurable and $\phi(s)\in F(s)$ for all $s\in T$. The set of all measurable selections of $F$ is denoted by $\mathcal{S}(F)$ and we write $\mathcal{S}^p(F):=\mathcal{S}(F)\cap L^p(T,X,\tau)$ where $L^p(T,X,\tau)$ is the $L^p$-space defined as above via Sturm's integral. The Aumann-Sturm integral of $F$ is then defined as
\[
\int F\,d\tau=\left\{\int \phi\,d\tau\mid \phi\in\mathcal{S}^1(F)\right\}.
\]

\section{Random monotone vector fields on nonlinear spaces}\label{sec:randOp}

We now bring the previous preliminaries together to define stochastically perturbed monotone vector fields on $\CAT$ spaces. For that, let $(T,\mathcal{T},\tau)$ be a probability space and $X$ be a separable Hadamard space (recall also Remark \ref{rem:tangentSepAndMeas}).

Consider a set-valued map $A:T\times X\to 2^{TX}$ with $A(s,x)\subseteq T_xX$ for all $s\in T$ and $x\in X$ such that $A(s,\cdot)$ is monotone for any $s\in T$.

For such a map, we define the resolvent similar to before via 
\[
J_\lambda(s,x):=\{z\in X\mid \tfrac{1}{\lambda}\log_zx\in A(s,z)\}
\]
for $\lambda>0$, $s\in T$ and $x\in X$. As before, if $A(s,\cdot)$ is monotone, then such a $z$ is necessarily unique if it exists and we denote it by $J_\lambda(s,x)$, similar to before.

The Yosida approximate is then lifted to this setting via
\[
A_\lambda(s,x):=\tfrac{1}{\lambda}\log_{J_\lambda(s,x)} x,
\]
so that Lemma \ref{YosidaProp} yields $A_\lambda(s,x)\in A(s,J_\lambda(s,x))$ as well as $\norm{A_\lambda(s,x)}_{J_\lambda(s,x)}=\frac{1}{\lambda}d(x,J_\lambda(s,x))$ for all $\lambda>0$, $s\in T$ and $x\in\mathrm{dom}(J_\lambda(s,\cdot))$.

There are various possible measurability properties for such maps which can be imposed. In the context of a Hilbert space with an operator $A:T\times X\to 2^X$, the most direct is perhaps to assume that $A$ satisfies
\begin{equation}
\{(s,x)\in T\times X\mid A(s,x)\cap U\neq\emptyset\}\in\mathcal{T}\otimes\mathcal{B}(X)\tag{\%}\label{meas}
\end{equation}
for any open set $U\subseteq X$, a property that is often called Effros measurability. We refer to \cite{AubinFrankowska2009,CastaingValadier1977} for further discussions on measurable set-valued maps. In Hilbert spaces and when $A$ is maximally monotone, as also outlined in \cite{Bianchi2016}, it follows by \cite[Lemma 2.1]{Attouch1979} that \eqref{meas} is equivalent to the property that $J_\lambda(\cdot,x)$ is $\mathcal{T}$/$\mathcal{B}(X)$-measurable for any $x\in X$ and some (or any) $\lambda>0$.

However, this equivalence does not seem to readily transfer to this metric setting. In particular, it is already not completely clear how to adequately transfer the above measurability assumption to the nonlinear context, and if phrased as satisfying \eqref{meas} for any open set $U\subseteq TX$, then the question for a suitable topology on $TX$ remains, where for any immediate such choice, neither direction of the above equivalence seems to hold.

As it will be critical for us to guarantee the measurability of the resolvent, we will hence focus on this latter property (similar to \cite{Bianchi2016}). So, we arrive at the following official definition:

\begin{definition}\label{def:randMonVecField}
A map $A:T\times X\to 2^{TX}$ with $A(s,x)\subseteq T_xX$ is called a random monotone vector field if $A(s,\cdot)$ is monotone for any $s\in T$ and $J_\lambda(\cdot,x)$ is $\mathcal{T}$/$\mathcal{B}(X)$-measurable for any $x\in X$ and any $\lambda>0$.
\end{definition}

Note that this property fully suffices for our purposes here. In particular, we will make no further measurability assumptions on such set-valued maps $A$ in the following.

As will be discussed in Examples \ref{randomSubDiffEx} and \ref{randomSaddleEx} below, the above condition can be immediately verified for the fundamental examples of subdifferentials of convex functions as well as saddle fields.

In that context of a random monotone vector field $A$, now assume that $A(s,\cdot)$ also satisfies the surjectivity condition for all $s\in T$, i.e.\ for any $\lambda>0$ and $x\in X$, there exists a $z\in X$ with $\tfrac{1}{\lambda}\log_zx\in A(s,z)$. Then, note that $J_\lambda(s,\cdot)$ is single-valued, total and nonexpansive as discussed before and therefore, $J_\lambda(s,\cdot)$ is uniformly continuous for any $s\in T$. So, $J_\lambda$ is a Carath\'eodory map\footnote{A function $g:T\times X\to X$ is called a Carath\'eodory map if $g(\cdot,x)$ is measurable for all $x\in X$ and $g(s,\cdot)$ is continuous for all $s\in T$.} whenever $A$ is a random monotone vector field with the surjectivity condition. In particular, we then have that $J_\lambda$ is $\mathcal{T}\otimes\mathcal{B}(X)$/$\mathcal{B}(X)$-measurable in that case (see e.g.\ \cite[Lemma 8.2.6]{AubinFrankowska2009}).

For such a random monotone vector field $A$, we define its mean $\underline{A}$ via
\[
\underline{A}(x):=\int A(s,x)\;d\tau(s)
\]
where the integral refers to the Aumann-Sturm integral as defined before, now on $T_xX$. If $X$ is a Hilbert-Hadamard space, then $\underline{A}$ is monotone, which follows from Jensen's inequality (recall Lemma \ref{nonlinJensen}) and that the pseudo-Riemannian metric is affine in this setting (recall Lemma \ref{affine}). Further, we introduce the notation $S_A(x):=\mathcal{S}(A(\cdot,x))$ and $S^p_A(x):=\mathcal{S}^p(A(\cdot,x))$ for $p\geq 1$, and write
\[
\mathcal{Z}_A(p):=\left\{x\in X\mid \exists \phi\in S^p_A(x)\left(\int\phi\,d\tau=0\right)\right\}.
\]
It should be noted that $\mathcal{Z}_A(p)\subseteq\mathcal{Z}_A(1)=\zer \underline{A}$.

Later on, we will in particular assume that $A(s,\cdot)$ is strongly monotone with a modulus $\alpha(s)$, where $\alpha:T\to (0,\infty)$ forms a measurable and integrable function such that $\int\alpha\,d\tau>0$. Similar to above, this implies that $\underline{A}$ is strongly monotone with modulus $\underline{\alpha}:=\int\alpha\;d\tau>0$ if $X$ is a Hilbert-Hadamard space.

\begin{example}\label{randomSubDiffEx}
Let $(T,\mathcal{T},\tau)$ be a complete probability space. In analogy to \cite{RockafellarWets1998}, let $f:T\times X\to (-\infty,+\infty]$ be a normal convex integrand, i.e.\ $f(s,\cdot)$ is proper, lower-semicontinuous for all $s\in T$ and $f$ is $\mathcal{T}\otimes\mathcal{B}(X)$-measurable. Then for its associated subdifferential
\[
\partial f(s,x):=\{u\in T_xX\mid f(s,y)\geq f(s,x)+g_x(u,\log_x y)\text{ for all }y\in X\}
\]
as in Example \ref{subDiffEx}, its resolvents are given, following Example \ref{resEx}, by the proximal maps
\[
\mathrm{prox}_{\lambda}^f(s,x):=\argmin_{y\in X}\left\{ f(s,y)+\frac{1}{2\lambda}d^2(x,y)\right\}.
\]
By adapting arguments from \cite{RockafellarWets1998}, it is then easy to see that $\mathrm{prox}_{\lambda}^f(\cdot,x)$ is $\mathcal{T}$/$\mathcal{B}(X)$-measurable, so that $\partial f$ is a random monotone vector field. Now, define $F(x):=\int f(s,x)\,d\tau(s)$ and assume that $F$ is proper. It further follows immediately that $F$ is convex and lower-semicontinuous. If we now in analogy to \cite{Bianchi2016} assume that
\[
\underline{\partial f}(x)=\int\partial f(s,x)\,d\tau(s)=\partial\int f(s,x)\,d\tau(s)=\partial F(x),
\]
then we in particular have $\zer \underline{\partial f}=\zer \partial F=\argmin F$ by Example \ref{subDiffEx}. Already over Hilbert or even Euclidean spaces, the above interchange of the integral and the subdifferential is a non-trivial property in stochastic convex analysis (sufficient conditions and related discussions can be found in e.g.\ \cite{RockafellarWets1982}). We do not know if and how the full property extends to Hadamard manifolds or spaces. The direction $\underline{\partial f}(x)\subseteq \partial F(x)$ however remains true in Hilbert-Hadamard spaces: Take $\phi\in S^1(\partial f(\cdot,x))$, that is $f(s,y)\geq f(s,x)+g_x(\phi(s),\log_x y)$ for any $s\in T$ and $y\in X$. Integrating the above property and using that $g_x$ is affine (recall Lemma \ref{affine}) together with Jensen's inequality (recall Lemma \ref{nonlinJensen}) yields $F(y)\geq F(x)+g_x\left(\int\phi\,d\tau,\log_x y\right)$ for any $y\in X$, and so $\int\phi\,d\tau\in \partial F(x)$. In particular, zeros of $\underline{\partial f}$ are always minimizers of $F$. Lastly, note that if it is further assumed that $f(s,\cdot)$ is strongly convex with modulus $\alpha(s)>0$ such that $\alpha$ is integrable with $\underline{\alpha}:=\int\alpha\,d\tau>0$, then $F$ is strongly convex with modulus $\underline{\alpha}$. Further, by Proposition \ref{strConToStrMon}, $\partial f(s,\cdot)$ is strongly monotone with modulus $\alpha(s)$ and $\partial F$ is strongly monotone with modulus $\underline{\alpha}$. So, over a Hilbert-Hadamard space or if $\underline{\partial f}=\partial F$, also $\underline{\partial f}$ is strongly monotone with modulus $\underline{\alpha}$.
\end{example}

\begin{example}\label{randomSaddleEx}
Let $(T,\mathcal{T},\tau)$ be a complete probability space and let $l:T\times X\times Y\to\mathbb{R}$ be a stochastic saddle function over two Hadamard spaces $X$ and $Y$, that is $l(s,\cdot,\cdot)$ is a saddle function for each $s\in T$ and $l$ is $\mathcal{T}\otimes\mathcal{B}(X)\otimes\mathcal{B}(Y)$-measurable. Then for its associated saddle field
\[
A_l(s,x,y):=\partial [-l(s,\cdot,y)](x)\times \partial [l(s,x,\cdot)](y)
\]
as in Example \ref{saddleEx}, its resolvents are given, following Example \ref{resSaddleEx}, by
\[
\mathrm{R}_{\lambda}^l(s,x,y):=\argminimax_{(z,w)\in X\times Y}\left\{ l(s,z,w)-\frac{1}{2\lambda}d^2(z,x)+\frac{1}{2\lambda}d^2(w,y)\right\},
\]
which can be shown to be measurable similar to $\mathrm{prox}^f_\lambda$, so that it also follows that $A_l$ is a random monotone vector field. Similar to Example \ref{randomSubDiffEx}, considering the mean saddle function $L(x,y):=\int l(s,x,y)\,d\tau(s)$ (assuming that it is real-valued), we at least always have $\zer\underline{A_l}\subseteq \minimax L$.
\end{example}

\section{A stochastic proximal point algorithm}\label{sec:randPPA}

Extending the previous work of Bianchi \cite{Bianchi2016}, we now consider a  stochastic variant of the proximal point algorithm. Let $(E,\mathcal{E},\mu)$ be a probability space and let $X$ be a separable Hilbert-Hadamard space (recall also Remark \ref{rem:tangentSepAndMeas}). Further, let $A:E\times X\to 2^{TX}$ be a random monotone vector field such that $A(s,\cdot)$ satisfies the surjectivity condition for any $s\in E$.

The stochastic proximal point method is now given by the iteration
\begin{equation}
x_{n+1}:=J_{\lambda_n}(\xi_{n+1},x_n)\tag{SPPA}\label{PPA}
\end{equation}
for a given starting value $x_0\in X$, a sequence of parameters $(\lambda_n)\subseteq (0,\infty)$ and a sequence $(\xi_{n+1})$ of random variables $\xi_{n+1}:\Omega\to E$ for an ambient probability space $(\Omega,\mathcal{F},\PP)$ over which the iteration takes place. Note that each $x_n$ is thereby an ($X$-valued) random variable as $J_{\lambda_n}$ is a Carath\'eodory map.

In terms of the parameters, we make the assumptions that 
\begin{equation}
(\lambda_n)\in\ell_+^2\setminus\ell_+^1\text{ and that }(\xi_{n+1})\text{ is i.i.d.\ with distribution }\mu,\tag{A0}\label{A0}
\end{equation}
where we write $\ell_+^p$ for the space of nonnegative $p$-summable sequences. To distinguish the two notions of integration we get from the two probability spaces, we use $\int$ to denote integrals over $(E,\mathcal{E},\mu)$ and $\EE$ to denote integrals over $(\Omega,\mathcal{F},\PP)$. In further terms of notation, we in the following write $\mathcal{F}_n:=\sigma(\xi_1,\dots,\xi_n)$ as well as $\EE_n[\cdot]$ as a shorthand for the conditional expectation $\EE[\cdot\mid \mathcal{F}_n]$. Also, all (in)equalities are understood to hold almost surely (over the suitable probability space), if not stated otherwise.

Now, in the remainder of this section, we will establish a (quantitative) convergence result for the above stochastic proximal point method in the context of a strong monotonicity assumption. Concretely, we in the following assume that 
\begin{equation}
A(s,\cdot)\text{ is strongly monotone with modulus }\alpha(s)>0\text{ such that }\int\alpha\,d\mu>0.\tag{A1}\label{A1}
\end{equation}
As discussed before, as we are working over a Hilbert-Hadamard space, this property in fact entails that $\underline{A}$, the mean of the fields $A(s,\cdot)$, is strongly monotone with modulus $\underline{\alpha}:=\int\alpha\,d\mu>0$. We assume w.l.o.g.\ that $\alpha(s)\leq 1$ for any $s\in E$.

Motivated by the assumptions of \cite[Theorem 4]{Bianchi2016}, we assume that there exists a (hence unique) zero $x^*$ of $\underline{A}$ which satisfies 
\begin{equation}
x^*\in\mathcal{Z}_A(2).\tag{A2}\label{A2}
\end{equation}
Further, we fix a $\phi^*\in S^2_A(x^*)$ with $\int\phi^*\,d\mu=0_{x^*}$. 

Our key analytical ingredients will be two central almost sure inequalities, modeled after inequalities established in \cite{Bianchi2016}. Both of these rely on the following independence property derived from Lemma \ref{indepMetric}.

\begin{lemma}\label{indepProp}
For any $n\in\mathbb{N}$, we have $\EE_n[g_{x^*}(\phi^*(\xi_{n+1}),\log_{x^*}x_n)]=0$.
\end{lemma}
\begin{proof}
Using Lemma \ref{indepMetric} with $u=\phi^*(\xi_{n+1})$, $v=\log_{x^*}x_n$ and $\mathcal{T}_0=\mathcal{F}_n$, we get that
\[
\EE_n[g_{x^*}(\phi^*(\xi_{n+1}),\log_{x^*}x_n)]
=g_{x^*}(\EE_n[\phi^*(\xi_{n+1})],\EE_n[\log_{x^*}x_n]),
\]
using that $x_n$ and so $\log_{x^*}x_n$ are $\mathcal{F}_n$-measurable and that $\xi_{n+1}$ and so $\phi^*(\xi_{n+1})$ are independent of $\mathcal{F}_n$. By Lemma \ref{indepLem}, the independence of $\xi_{n+1}$ from $\mathcal{F}_n$ yields that $\EE_n[\phi^*(\xi_{n+1})]=\EE[\phi^*(\xi_{n+1})]$ (where one should note that $\phi^*(\xi_{n+1})\in L^2(\Omega,T_{x^*}X,\PP)$ as $\phi^*\in S^2_A(x^*)$). Further, by Lemma \ref{transfer} we have $\EE[\phi^*(\xi_{n+1})]=\int\phi^*\,d\mu=0_{x^*}$ and so $g_{x^*}(\EE_n[\phi^*(\xi_{n+1})],\EE_n[\log_{x^*}x_n])=0$. This yields the claim.
\end{proof}

\begin{remark}
Towards our main convergence result, we only use the assumption that $X$ is a Hilbert-Hadamard space to derive the above Lemma \ref{indepProp}. As such, our result extends to non-Hilbert-Hadamard spaces where only the tangent space $T_{x^*}X$ at the solution $x^*$ has flat curvature. This might be of help when one tries to apply the present results to spaces which are not fully Hilbert-Hadamard spaces, such as certain metric trees.
\end{remark}

The first of the two inequalities is now given in Lemma \ref{ineq1}, establishing a type of stochastic quasi-Fej\'er monotonicity for the iteration in question. This inequality is modeled after the proof of \cite[Proposition 1]{Bianchi2016} (see in particular p.\ 2244 therein).\footnote{The expression $\norm{A_{\lambda_n}(\xi_{n+1},x_n)}^2_{x_{n+1}}$ featuring in Lemma \ref{ineq1} could prove problematic from a measurability point of view, as it involves a random variable in the index of the tangent space norm, i.e.\ as the point of issue of a tangent space. However, note that $\lambda^2_n\norm{A_{\lambda_n}(\xi_{n+1},x_n)}_{x_{n+1}}^2=d^2(x_n,x_{n+1})$ as also crucially used in the proof, so that this expression is immediately measurable by the measurability of the resolvent and the joint measurability of the metric, using separability of $X$.}

\begin{lemma}\label{ineq1}
Let $\beta\in (0,\frac{1}{2}]$. For any $n\in\mathbb{N}$, we have
\[
\EE_n[d^2(x_{n+1},x^*)]\leq d^2(x_n,x^*)-\lambda_n^2(1-2\beta)\EE_n[\norm{A_{\lambda_n}(\xi_{n+1},x_n)}^2_{x_{n+1}}]+\lambda_n^2\frac{\int\norm{\phi^*}_{x^*}^2\,d\mu}{2\beta}
\]
\end{lemma}
\begin{proof}
If not stated otherwise, all equalities and inequalities are understood to hold almost surely (if applicable). Using Lemma \ref{tangentCat0}, we get
\[
d^2(x_{n+1},x^*)+d^2(x_{n+1},x_n)-2g_{x_{n+1}}(\log_{x_{n+1}} x_n,\log_{x_{n+1}} x^*)\leq d^2(x_n,x^*).
\]
Now, note that
\begin{align*}
g_{x_{n+1}}(\log_{x_{n+1}} x_n,\log_{x_{n+1}} x^*)&=\lambda_ng_{x_{n+1}}(\tfrac{1}{\lambda_n}\log_{x_{n+1}} x_n,\log_{x_{n+1}} x^*)\\
&=\lambda_n g_{x_{n+1}}(A_{\lambda_n}(\xi_{n+1},x_n),\log_{x_{n+1}} x^*)\\
&\leq -\lambda_n g_{x^*}(\phi^*(\xi_{n+1}),\log_{x^*} x_{n+1}),
\end{align*}
where the last inequality used the monotonicity of $A$, and that
\[
A_{\lambda_n}(\xi_{n+1},x_n)\in A(\xi_{n+1},J_{\lambda_n}(\xi_{n+1},x_n))=A(\xi_{n+1},x_{n+1})
\]
as well as $\phi^*(\xi_{n+1})\in A(\xi_{n+1},x^*)$. Now, note that
\begin{align*}
&-\lambda_n g_{x^*}(\phi^*(\xi_{n+1}),\log_{x^*} x_{n+1})\\
&\qquad=-\lambda_n g_{x^*}(\phi^*(\xi_{n+1}),\log_{x^*} x_{n})\\
&\qquad\phantom{=}\;+\lambda_n (g_{x^*}(\phi^*(\xi_{n+1}),\log_{x^*} x_{n})- g_{x^*}(\phi^*(\xi_{n+1}),\log_{x^*} x_{n+1}))\\
&\qquad\leq -\lambda_n g_{x^*}(\phi^*(\xi_{n+1}),\log_{x^*} x_{n})+\lambda_n\norm{\phi^*(\xi_{n+1})}_{x^*}d_{x^*}(\log_{x^*} x_n,\log_{x^*} x_{n+1})\\
&\qquad\leq -\lambda_n g_{x^*}(\phi^*(\xi_{n+1}),\log_{x^*} x_{n})+\lambda_n\norm{\phi^*(\xi_{n+1})}_{x^*}d(x_n,x_{n+1})
\end{align*}
using Lemma \ref{nonexpCat0} and the fact that $\log_{x^*}$ is nonexpansive in a Hadamard space (recall Lemma \ref{lognonexp}). Using (analogously to the proof of Lemma 2 in \cite{Bianchi2016}) that
\[
\lambda_n\norm{\phi^*(\xi_{n+1})}_{x^*}d(x_n,x_{n+1})\leq \frac{\lambda_n^2}{4\beta}\norm{\phi^*(\xi_{n+1})}_{x^*}^2+\beta d^2(x_n,x_{n+1}),
\]
we get that
\begin{align*}
&g_{x_{n+1}}(\log_{x_{n+1}} x_n,\log_{x_{n+1}} x^*)\\
&\qquad\leq -\lambda_n g_{x^*}(\phi^*(\xi_{n+1}),\log_{x^*} x_{n})+\frac{\lambda_n^2}{4\beta}\norm{\phi^*(\xi_{n+1})}_{x^*}^2+\beta d^2(x_n,x_{n+1})
\end{align*}
and so, using that $\lambda^2_n\norm{A_{\lambda_n}(\xi_{n+1},x_n)}_{x_{n+1}}^2=d^2(x_n,x_{n+1})$, we get
\begin{align*}
d^2(x_{n+1},x^*)\leq\; &d^2(x_n,x^*)-\lambda_n^2(1-2\beta)\norm{A_{\lambda_n}(\xi_{n+1},x_n)}_{x_{n+1}}^2\\
&-2\lambda_n g_{x^*}(\phi^*(\xi_{n+1}),\log_{x^*} x_{n})+\frac{\lambda_n^2}{2\beta}\norm{\phi^*(\xi_{n+1})}_{x^*}^2.
\end{align*}
We now apply the conditional expectation $\EE_n$. Using the usual transformation rule, we get that $\EE_n[\norm{\phi^*(\xi_{n+1})}_{x^*}^2]= \int\norm{\phi^*}_{x^*}^2\,d\mu$. We get $\EE_n[g_{x^*}(\phi^*(\xi_{n+1}),\log_{x^*} x_{n})]=0$ using Lemma \ref{indepProp}. This yields
\[
\EE_n[d^2(x_{n+1},x^*)]\leq d^2(x_n,x^*)-\lambda_n^2(1-2\beta)\EE_n[\norm{A_{\lambda_n}(\xi_{n+1},x_n)}_{x_{n+1}}^2]+\lambda_n^2\frac{\int\norm{\phi^*}_{x^*}^2\,d\mu}{2\beta}
\]
which was the claim.
\end{proof}

Importantly, the above inequality yields the square integrability of the sequence $(x_n)$:

\begin{corollary}\label{L2Bound}
$\EE[d^2(x_n,x^*)]\leq d^2(x_0,x^*)+\int\norm{\phi^*}_{x^*}^2\,d\mu\sum_{n=0}^\infty\lambda_n^2<\infty$ for any $n\in \mathbb{N}$.
\end{corollary}

The next inequality again establishes a type of stochastic quasi-Fej\'er monotonicity for the iteration in question, however now with a different selection of error terms in the recurrence inequality based on the strong monotonicity of the field which now plays a crucial role, compared to the former inequality where it was not used. While at first sight perhaps redundant, it is exactly the interplay between this and the former inequality that will allow us to establish rates of convergence of the iteration in the end. This inequality is modeled after the proof of \cite[Theorem 4]{Bianchi2016} (see in particular p.\ 2253 therein).

\begin{lemma}\label{ineq2}
For any $n\in\mathbb{N}$, we have
\[
\EE_n[d^2(x_{n+1},x^*)]\leq(1+2\lambda_n^2)d^2(x_n,x^*)-2\lambda_n\underline{\alpha}d^2(x_n,x^*)+\lambda_n^2V_n
\]
for $V_n=2\EE_n[\norm{A_{\lambda_n}(\xi_{n+1},x_n)}_{x_{n+1}}^2]+\int\norm{\phi^*}^2_{x^*}\,d\mu$.
\end{lemma}
\begin{proof}
We proceed similarly to the proof of Lemma \ref{ineq1} and get
\[
d^2(x_{n+1},x^*)+d^2(x_{n+1},x_n)-2g_{x_{n+1}}(\log_{x_{n+1}} x_n,\log_{x_{n+1}} x^*)\leq d^2(x_n,x^*)
\]
as before using Lemma \ref{tangentCat0}. Applying strong monotonicity in place of monotonicity then yields
\begin{align*}
&g_{x_{n+1}}(\log_{x_{n+1}} x_n,\log_{x_{n+1}} x^*)\\
&\qquad\leq -\lambda_n g_{x^*}(\phi^*(\xi_{n+1}),\log_{x^*} x_{n+1})-\lambda_n\alpha(\xi_{n+1})d^2(x_{n+1},x^*).
\end{align*}
As before in the proof of Lemma \ref{ineq1} (now with $\beta=\frac{1}{2}$), we get
\begin{align*}
&-\lambda_n g_{x^*}(\phi^*(\xi_{n+1}),\log_{x^*} x_{n+1})\\
&\qquad\leq -\lambda_n g_{x^*}(\phi^*(\xi_{n+1}),\log_{x^*} x_{n})+\frac{\lambda_n^2}{2}\norm{\phi^*(\xi_{n+1})}_{x^*}^2+\frac{1}{2} d^2(x_n,x_{n+1})
\end{align*}
and so
\begin{align*}
d^2(x_{n+1},x^*)\leq\; &d^2(x_n,x^*)-2\lambda_n\alpha(\xi_{n+1})d^2(x_{n+1},x^*)\\
&-2\lambda_n g_{x^*}(\phi^*(\xi_{n+1}),\log_{x^*} x_{n})+\lambda_n^2\norm{\phi^*(\xi_{n+1})}_{x^*}^2.
\end{align*}
Now, using Lemma \ref{tangentCat0} again, we get
\begin{align*}
d^2(x_{n+1},x^*)&\geq d^2(x_n,x_{n+1})+d^2(x_n,x^*)-2g_{x_n}(\log_{x_n} x_{n+1},\log_{x_n} x^*)\\
&\geq d^2(x_n,x^*)-2\lambda_ng_{x_n}(\tfrac{1}{\lambda_n}\log_{x_n} x_{n+1},\log_{x_n}x^*)\\
&\geq d^2(x_n,x^*)-\lambda_n\norm{\tfrac{1}{\lambda_n}\log_{x_n} x_{n+1}}_{x_n}^2-\lambda_n \norm{\log_{x_n} x^*}_{x_n}^2
\end{align*}
where the third inequality follows using the definition of $g_{x_n}$. Combined, this yields
\begin{gather*}
d^2(x_{n+1},x^*)\leq d^2(x_n,x^*)-2\lambda_n\alpha(\xi_{n+1})d^2(x_n,x^*)+2\lambda_n^2\norm{\tfrac{1}{\lambda_n}\log_{x_n} x_{n+1}}_{x_n}^2\\
+2\lambda_n^2 \norm{\log_{x_n} x^*}_{x_n}^2-2\lambda_n g_{x^*}(\phi^*(\xi_{n+1}),\log_{x^*} x_{n})+\lambda_n^2\norm{\phi^*(\xi_{n+1})}_{x^*}^2,
\end{gather*}
where we have in particular used that $\alpha(s)\leq 1$. Note now that $\norm{\log_{x_n} x^*}_{x_n}=d(x_n,x^*)$ and that $\lVert\tfrac{1}{\lambda_n}\log_{x_n} x_{n+1}\rVert_{x_n}=\tfrac{1}{\lambda_n}d(x_n,x_{n+1})=\norm{A_{\lambda_n}(\xi_{n+1},x_n)}_{x_{n+1}}$, so that we have
\begin{gather*}
d^2(x_{n+1},x^*)\leq  (1+2\lambda_n^2)d^2(x_n,x^*)-2\lambda_n\alpha(\xi_{n+1})d^2(x_n,x^*)\\
+2\lambda_n^2\norm{A_{\lambda_n}(\xi_{n+1},x_n)}_{x_{n+1}}^2-2\lambda_n g_{x^*}(\phi^*(\xi_{n+1}),\log_{x^*} x_{n})+\lambda_n^2\norm{\phi^*(\xi_{n+1})}_{x^*}^2.
\end{gather*}
We now again apply the conditional expectation $\EE_n$. Using the usual transformation rule, we get $\EE_n[\norm{\phi^*(\xi_{n+1})}_{x^*}^2]= \int\norm{\phi^*}_{x^*}^2\,d\mu$ and $\EE_n[g_{x^*}(\phi^*(\xi_{n+1}),\log_{x^*} x_{n})]=0$ follows from Lemma \ref{indepProp}, both as before. Also, using the independence of $\xi_{n+1}$ and $x_n$ as well as the usual transformation rule yields $\EE_n[\alpha(\xi_{n+1})d^2(x_n,x^*)]=\underline{\alpha}d^2(x_n,x^*)$. Combined, we have
\[
\EE_n[d^2(x_{n+1},x^*)]\leq(1+2\lambda_n^2)d^2(x_n,x^*)-2\lambda_n\underline{\alpha}d^2(x_n,x^*)+\lambda_n^2V_n
\]
which was the claim.
\end{proof}

Next, we endow our qualitative assumptions on the parameter sequence $(\lambda_n)$ with moduli witnessing their quantitative content. To be precise, we in the following assume that we have functions $\chi:(0,\infty)\to\mathbb{N}$ and $\theta:\mathbb{N}\times (0,\infty)\to\mathbb{N}$ such that 
\begin{enumerate}
\item $\sum_{n=\chi(\varepsilon)}^\infty \lambda_n^2<\varepsilon$ for all $\varepsilon>0$,
\item $\sum_{n=k}^{\theta(k,b)}\lambda_n\geq b$ for all $b>0$ and $k\in\mathbb{N}$.
\end{enumerate}
We also assume a bound $\Lambda > \sum_{n=0}^\infty \lambda_n^2$ and that we are given a $c>0$ with $c>\int\norm{\phi^*}_{x^*}^2\,d\mu$. Lastly, we assume that $b>0$ satisfies $b>d^2(x_0,x^*)$.

The second-to-last key quantitative result that we quote from the literature is the following quantitative version of a lemma of Qihou \cite{Qihou2001} (see also \cite[Lemma 5.31]{BauschkeCombettes2017}):

\begin{lemma}[{see \cite[Theorem 3.2]{NeriPowell2024}}]\label{qihou}
Let $(x_n)$, $(\alpha_n)$, $(\beta_n)$ and $(\gamma_n)$ be sequences of nonnegative reals with
\[
x_{n+1}\leq (1+\alpha_n)x_n-\beta_n+\gamma_n
\]
for all $n\in\mathbb{N}$. If $\prod_{i=0}^\infty(1+\alpha_i)<\infty$ and $\sum_{i=0}^\infty\gamma_i<\infty$, then $(x_n)$ converges and $\sum_{i=0}^\infty\beta_i<\infty$.

Further, if $K,L,M>0$ satisfy $x_0<K$, $\prod_{i=0}^\infty(1+\alpha_i)<L$ and $\sum_{i=0}^\infty\gamma_i<M$, then $\sum_{i=0}^\infty\beta_i<L(K+M)$.
\end{lemma}

Finally, we only require the following rather immediate result (which can for example be found in \cite{NeriPischkePowell2025}). 

\begin{lemma}\label{sumconv}
Suppose that $(u_n)$, $(v_n)$ are sequences of nonnegative reals with $L>0$ such that $\sum_{n=0}^\infty u_nv_n< L$ and $\theta:\mathbb{N}\times(0,\infty)\to \mathbb{N}$ such that $\sum_{n=k}^{\theta(k,b)} u_n\geq b$ for all $b>0$ and $k\in\mathbb{N}$. Then $\liminf_{n\to\infty}v_n=0$ with 
\[
\forall\varepsilon>0\ \forall N\in\mathbb{N}\ \exists n\in [N;\theta(N,L/\varepsilon)](v_n<\varepsilon).
\]
\end{lemma}
\begin{proof}
For arbitrary $\varepsilon>0$ and $N\in\mathbb{N}$, suppose for a contradiction that $v_n\geq\varepsilon$ for all $n\in [N;\theta(N,L/\varepsilon)]$. Then $L\leq\varepsilon\sum_{n=N}^{\theta(N,L/\varepsilon)}u_n\leq \sum_{n=N}^{\theta(N,L/\varepsilon)} u_nv_n\leq \sum_{n=0}^{\infty} u_nv_n < L$, which is a contradiction.
\end{proof}

We can now employ this to derive a so-called $\liminf$-rate in expectation for the sequence $d^2(x_n,x^*)$:

\begin{lemma}\label{liminf}
It holds that $\liminf_{n\to\infty}\EE[d^2(x_n,x^*)]=0$, with 
\[
\forall\varepsilon>0\ \forall N\in\mathbb{N}\ \exists n\in [N;\theta(N,D/\varepsilon)](\EE[d^2(x_n,x^*)]<\varepsilon)
\]
where $C:=8\left(b+\Lambda2c\right)+\Lambda c$ and $D:=e^{2\Lambda}(b+C)/2\underline{\alpha}$.
\end{lemma}
\begin{proof}
By Lemma \ref{ineq1} with $\beta=\tfrac{1}{4}$, we have
\[
\EE[d^2(x_{n+1},x^*)]\leq \EE[d^2(x_n,x^*)]-\frac{\lambda_n^2}{2}\EE[\norm{A_{\lambda_n}(\xi_{n+1},x_n)}_{x_{n+1}}^2]+\lambda_n^22c.
\]
Applying Lemma \ref{qihou} yields $\sum_{n=0}^\infty\frac{\lambda_n^2}{2}\EE[\norm{A_{\lambda_n}(\xi_{n+1},x_n)}_{x_{n+1}}^2]<2\left(b+\Lambda 2c\right)$ and so
\[
\sum_{n=0}^\infty\lambda_n^2\left(2\EE[\norm{A_{\lambda_n}(\xi_{n+1},x_n)}_{x_{n+1}}^2]+c\right)<8\left(b+\Lambda 2c\right)+\Lambda c=:C.
\]
By Lemma \ref{ineq2}, we have
\begin{align*}
&\EE[d^2(x_{n+1},x^*)]\\
&\;\;\;\leq(1+2\lambda_n^2)\EE[d^2(x_n,x^*)]-2\lambda_n\underline{\alpha}\EE[d^2(x_n,x^*)]+\lambda_n^2\left(2\EE[\norm{A_{\lambda_n}(\xi_{n+1},x_n)}_{x_{n+1}}^2]+c\right)
\end{align*}
and so Lemma \ref{qihou} implies that $\sum_{n=0}^\infty\lambda_n\EE[d^2(x_n,x^*)]<\frac{e^{2\Lambda}(b+C)}{2\underline{\alpha}}=:D$. Finally, Lemma \ref{sumconv} implies $\liminf_{n\to\infty}\EE[d^2(x_n,x^*)]=0$ together with the respective rate as claimed. 
\end{proof}

We can now already present our main theorem:

\begin{theorem}\label{thm:main}
Let $(E,\mathcal{E},\mu)$ and $(\Omega,\mathcal{F},\PP)$ be probability spaces. Let $X$ be a separable Hilbert-Hadamard space. Let $A:E\times X\to 2^{TX}$ with $A(s,x)\subseteq T_xX$ be a random monotone vector field and assume that $A(s,\cdot)$ satisfies the surjectivity condition for any $s\in E$. Let $(x_n)$ be the iteration given by \eqref{PPA}. Assume \eqref{A0} -- \eqref{A2}. Then it holds that 
\[
\EE[d^2(x_n,x^*)]\to 0\text{ and }d^2(x_n,x^*)\to 0\text{ a.s.}
\]
Moreover, the following rates of convergence apply: Let $\chi:(0,\infty)\to\mathbb{N}$ and $\theta:\mathbb{N}\times (0,\infty)\to\mathbb{N}$ be such that 
\[
\forall \varepsilon>0\left(\sum_{n=\chi(\varepsilon)}^\infty \lambda_n^2<\varepsilon\right) \text{ and }\forall b>0\ \forall k\in\mathbb{N}\left(\sum_{n=k}^{\theta(k,b)}\lambda_n\geq b\right).
\]
Let $\Lambda > \sum_{n=0}^\infty \lambda_n^2$. Also, let $\phi^*\in S^2_A(x^*)$ such that $\int\phi^*\,d\mu=0_{x^*}$ and $c>0$ with $c>\int\norm{\phi^*}^2_{x^*}\,d\mu$. Lastly, let $b>0$ be such that $b>d^2(x_0,x^*)$. Then
\[
\forall \varepsilon>0\ \forall n\geq \rho(\varepsilon)\left(\EE[d^2(x_n,x^*)]<\varepsilon\right)
\]
with rate $\rho(\varepsilon):=\theta(\chi(\varepsilon/2c),2D/\varepsilon)$ and 
\[
\forall \lambda,\varepsilon>0\left(\PP\left(\exists n\geq\rho'(\lambda,\varepsilon)\left(d^2(x_n,x^*)\geq\varepsilon\right)\right)<\lambda\right)
\]
with rate $\rho'(\lambda,\varepsilon):=\rho(\lambda\varepsilon)$. Here: $C:=8\left(b+\Lambda2c\right)+\Lambda c$ and $D:=e^{2\Lambda}(b+C)/2\underline{\alpha}$.
\end{theorem}
\begin{proof}
It obviously suffices to establish the quantitative results. For any $n\in\mathbb{N}$, define $X_{n}:=d^2(x_n,x^*)+c\sum^\infty_{k=n}\lambda_k^2$. As $(x_n)$ is adapted to $(\mathcal{F}_n)$, also $(X_{n})$ is adapted to $(\mathcal{F}_n)$. As we have
\[
\EE_n[d^2(x_{n+1},x^*)]\leq d^2(x_n,x^*)+\lambda_n^2c
\]
by Lemma \ref{ineq1}, the stochastic process $(X_{n})$ is a nonnegative supermartingale. Indeed, note that 
\begin{align*}
\EE_n[X_{n+1}]&=\EE_n\left[d^2(x_{n+1},x^*)\right]+c\sum^\infty_{k={n+1}}\lambda_k^2\leq d^2(x_n,x^*) + c\sum^\infty_{k=n}\lambda_k^2=X_{n}.
\end{align*}
Now, let $\varepsilon>0$ be arbitrary. Using Lemma \ref{liminf}, we choose an
\[
n_0\in \left[\chi(\varepsilon/2c);\theta(\chi(\varepsilon/2c),2D/\varepsilon)\right]
\]
such that $\EE[d^2(x_{n_0},x^*)]<\varepsilon/2$. Let $n\geq n_0$ be arbitrary. Then
\[
\EE[d^2(x_{n},x^*)]\leq \EE[X_{n}]\leq \EE[X_{n_0}]= \EE[d^2(x_{n_0},x^*)]+c\sum^\infty_{k=n_0}\lambda_k^2< \varepsilon
\]
using that $(X_{n})$ is a supermartingale and the properties of $\chi$. As $n$ was arbitrary, this yields $\EE[d^2(x_{n},x^*)]\to 0$ and that $\rho$ is a rate of convergence for that limit. For $d^2(x_{n},x^*)\to 0$ a.s., note that
\begin{align*}
\PP(\exists n\geq n_0(d^2(x_{n},x^*)\geq a))\leq \PP(\exists n\geq n_0(X_n\geq a))\leq \frac{\EE[X_{n_0}]}{a}
\end{align*}
where the second inequality follows from Ville's inequality \cite{Ville1939} (see also \cite{Metivier1982}). This immediately implies that $d^2(x_{n},x^*)\to 0$ a.s.\ with rate $\rho'$.
\end{proof}

While the proof is presented here in a self-contained style, it follows the outline and is effectively an instance of the proof of \cite[Theorem 2.8]{NeriPischkePowell2025}.

To our knowledge, in the generality presented here, the quantitative aspects of the above Theorem \ref{thm:main} are already novel in the Hilbert space context. The qualitative convergence result generalizes \cite[Theorem 4]{Bianchi2016} and in that way seems to be novel in all classes of Hilbert-Hadamard spaces transcending Hilbert spaces, in particular including Hadamard manifolds.

\begin{remark}
If $\rho$ is invertible and decreasing, Theorem \ref{thm:main} immediately implies the nonasymptotic guarantee $\EE[d^2(x_n,x^*)]\leq\rho^{-1}(n)$ for all $n\in\mathbb{N}$. We can then also derive a similar estimate for $\PP(\exists m\geq n(d^2(x_{m},x^*)\geq \varepsilon))$. However, the complexity of the rates that arise from Theorem \ref{thm:main} is without further assumptions rather dire, namely exponential:\footnote{As such, the rates are similar for deterministic variants of the proximal point algorithm in metric settings without further assumptions than strong monotonicity, as e.g.\ obtained in \cite{LeusteanSipos2018} for the special case of strongly (even uniformly) convex functions.} For the canonical choice of $\lambda_n=\frac{1}{n+1}$, a quick calculation shows that we get
\[
\EE[d^2(x_n,x^*)]\leq\frac{4\max\{C,D\}}{\ln(n+2)}\text{ for all }n\in\mathbb{N}
\]
in that case, with a similar bound on $\PP(\exists m\geq n(d^2(x_{m},x^*)\geq \varepsilon))$. This is because we make no other assumptions on $A$ besides the strong monotonicity assumption and the minor underlying measurability assumptions. A brief discussion on fast rates under additional assumptions is given below.
\end{remark}

As a particular corollary, we get the following result on minimizing strongly convex integrands as discussed in Example \ref{randomSubDiffEx}:

\begin{corollary}\label{cor:convexFct}
Let $(E,\mathcal{E},\mu)$ and $(\Omega,\mathcal{F},\PP)$ be probability spaces. Let $X$ be a separable Hilbert-Hadamard space. Let $f:E\times X\to (-\infty,+\infty]$ be a normal convex integrand such that $f(s,\cdot)$ is strongly convex with modulus $\alpha(s)>0$ and $\alpha$ is integrable with $\int\alpha\,d\mu>0$. Assume that $F(x):=\int f(s,x)\,d\mu(s)$ is proper. Let $(x_n)$ be the iteration given by $x_{n+1}:=\mathrm{prox}^f_{\lambda_n}(\xi_{n+1},x_n)$, and assume \eqref{A0} and \eqref{A2}.

Then $\EE[d^2(x_n,x^*)]\to 0$ and $d^2(x_n,x^*)\to 0$ a.s., and $x^*$ is the unique minimizer of $F$. Moreover, rates of convergence can be computed for these limits in similarity to Theorem \ref{thm:main}.
\end{corollary}

Note that assumption \eqref{A1} is immediately satisfied in the context of Corollary \ref{cor:convexFct} by virtue of Proposition \ref{strConToStrMon}.

Similarly, we get a corollary for saddle point problems but we do not spell this out in detail here.

As a last result, we briefly discuss an additional assumption on $A$ that allows us to derive fast rates of convergence. Concretely, assume the following uniform boundedness property for the second moments of the Yosida approximates along the iteration: 
\begin{equation}
\text{There is a $\sigma>0$ such that }\EE[\norm{A_{\lambda_n}(\xi_{n+1},x_n)}^2_{x_{n+1}}]\leq\sigma\text{ for all }n\in\mathbb{N}.\tag{A4}\label{A4}
\end{equation}
This is a special case of a general uniform boundedness assumption of second moments of \emph{arbitrary} selections from the random field $A$. In that way, assumption \eqref{A4} is akin to the uniform boundedness assumptions for second moments of subdifferential selections as considered over linear spaces, e.g., in \cite{AsiDuchi2019,EisenmannStillfjordWilliamson2022}. In that context, we can derive the following fast nonasymptotic guarantees:

\begin{theorem}\label{thm:mainFast}
In the context of Theorem \ref{thm:main}, assume \eqref{A4}. Then for $\lambda_{n}:=1/\underline{\alpha}(n+2)$, it holds that
\[
\EE[d^2(x_n,x^*)]\leq \frac{u}{n+2}\text{ and }\PP\left(\exists m\geq n\left(d^2(x_m,x^*)\geq\varepsilon\right)\right)\leq \frac{1}{\varepsilon}\cdot \frac{2u}{n+2}
\]
for all $n\in\mathbb{N}$, where $u=\max\{(4\sigma+2c)/\underline{\alpha}^2,\ceil*{4/\underline{\alpha}^2}(b+c\Lambda)\}$  and $b\geq d^2(x_{0},x^*)$.
\end{theorem}
\begin{proof}
Note that $\EE[d^2(x_n,x^*)]< b+c\Lambda$ by Corollary \ref{L2Bound}. Hence, the claim is clear for $n\leq n_0:=\ceil*{4/\underline{\alpha}^2}-2$. Using \eqref{A4}, it follows from Lemma \ref{ineq2} that
\[
\EE[d^2(x_{n+1},x^*)]\leq(1+2\lambda_n^2-2\lambda_n\underline{\alpha})\EE[d^2(x_n,x^*)]+\lambda_n^2(2\sigma+c).
\]
It is straightforward, albeit a bit tedious, to verify that $1+2\lambda_n^2-2\lambda_n\underline{\alpha}\leq 1-1.5/(n+2)
$ for $n\geq n_0$, so that we obtain
\[
\EE[d^2(x_{n+1},x^*)]\leq\left(1-\frac{1.5}{n+2}\right)\EE[d^2(x_n,x^*)]+\lambda_n^2(2\sigma+c)
\]
for all such $n$. We then get $\EE[d^2(x_{n},x^*)]\leq u/(n+2)$ for the $u$ above by induction on $n$ (see also \cite[Lemma 3.5]{NeriPischkePowell2025}). Akin to the proof of Theorem \ref{thm:main}, we can then also derive
\begin{align*}
\PP(\exists m\geq n(d^2(x_{m},x^*)\geq \varepsilon))\leq \frac{1}{\varepsilon}\cdot \frac{2u}{n+2}
\end{align*}
by first moving to the supermartingale $X_n=d^2(x_n,x^*)+c\sum_{k=n}^\infty\lambda_k^2$, showing $\EE[X_n]\leq 2u/(n+2)$ and then applying Ville's inequality. We omit the details here.
\end{proof}

As with the proof of Theorem \ref{thm:main}, while we have presented the proof in a style tailored to \eqref{PPA}, it follows the outline and is effectively an instance of the proof of \cite[Theorem 3.6]{NeriPischkePowell2025}.\\

\noindent {\bf Acknowledgments:} I want to thank Ulrich Kohlenbach, Morenikeji Neri and in particular Thomas Powell for many helpful comments on a previous draft of this paper. I also want to thank the anonymous referees for the very careful reading of the manuscript and the many helpful suggestions which greatly improved the paper at various places.

\bibliographystyle{plain}
\bibliography{ref}

\end{document}